\documentclass[a4paper,10pt,reqno]{amsart}
\usepackage{amsthm,amssymb,latexsym,amsmath}

\usepackage[latin1]{inputenc}
\usepackage{amsmath,amssymb,amsfonts,amsthm}
\usepackage{nth}

\newtheorem{theorem}{Theorem}[section]

\renewcommand{\colon}{\nobreak\mskip2mu\mathpunct{}\nonscript
  \mkern-\thinmuskip{:}\mskip6muplus1mu\relax}

\newtheorem{proposition}[theorem]{Proposition}
\newtheorem{corollary}[theorem]{Corollary}
\newtheorem{lemma}[theorem]{Lemma}
\theoremstyle{definition}
\newtheorem{definition}[theorem]{Definition}
\newtheorem{example}[theorem]{Example}

\newtheorem{remark}[theorem]{Remark}
\numberwithin{equation}{section}
\usepackage{color}

\def\ceq{\mathrel{\vcenter{\hbox{:}}{=}}}

\newcommand{\opn}{{\mathcal O}_{\mathbb{P}^2}}

\newcommand{\F}{{\mathcal{F}}}
\newcommand{\G}{{\mathcal{G}}}

\newcommand{\Q}{\mathbb{Q}}

\newcommand{\C}{\mathbb{C}}

\newcommand{\OO}{\mathcal{O}}

\newcommand{\PP}{\mathbb P}

\newcommand{\Z}{\mathbb Z}
\newcommand{\aut}{\mbox{Aut}}
\newcommand{\bim}{\mbox{Bim}}
\newcommand{\tg}{\mbox{tang}}
\newcommand{\HH}{\mathrm{H}}
\newcommand{\gl}{\mbox{GL}}
\newcommand{\pgl}{\mbox{PGL}}
\newcommand{\sing}{\mbox{Sing}}
\newcommand{\kod}{\mbox{Kod}}
\newcommand{\lcm}{\mbox{lcm}}
\newcommand{\di}{\mbox{d}}
\newcommand{\dx}{\mbox{dx}}

\newcommand{\dz}{\mbox{dz}}
\newcommand{\dw}{\mbox{dw}}
\newcommand{\dX}{\mbox{dX}}
\newcommand{\dY}{\mbox{dY}}
\newcommand{\dZ}{\mbox{dZ}}

\newcommand{\pder}[2]{\frac{\partial #1}{\partial #2}}

\title[Polynomial bounds for automorphisms groups of foliations]{Polynomial bounds for automorphisms groups of foliations}
\date{\today}

\author[Corr\^ea, Muniz]
{M. Corr\^ea \and A. Muniz }

\address{\emph{M. Corr\^ea }: Depto. de Mat.--ICEX
 Universidade Federal de Minas Gerais, UFMG}
\curraddr{Av. Ant\^onio Carlos 6627, 31270-901,
 Belo Horizonte-MG, Brasil.}
\email{mauriciomatufmg@gmail.com}

\address{\emph{A. Muniz}: Depto. de Mat.--ICEX
 \\ Universidade Federal de Minas Gerais, UFMG}
\curraddr{Av. Ant\^onio Carlos 6627, 31270-901,
 Belo Horizonte-MG, Brasil.}
\email{alannmuniz@gmail.com}


\begin{document}

\begin{abstract}
Let $(X, \F)$ be a foliated surface and $G$ a finite group of
automorphisms of $X$ that preserves $\F$. We investigate invariant
loci for $G$ and obtain upper bounds for its order that depends 
polynomially on the Chern numbers of $X$ and $\F$. As a consequence,
we estimate the order of the automorphism group of some foliations
under mild restrictions. We obtain an optimal bound for foliations on the projective plane which  is attained by the automorphism groups of the Jouanolou's foliations.
\end{abstract}

\maketitle
%\tableofcontents

\section{Introduction}

The problem of bounding automorphisms traces back to the late \nth{19} century in the works of Felix Klein, Adolf Hurwitz and others on the study of Riemann surfaces. In 1882, Klein proved that a compact Riemann surface $C$ of genus $g\geq 2$ has finite automorphism group and Hurwitz, in 1892, gave a upper bound to the order of such groups. Lower bounds were found under stronger assumptions. In 1895, Wiman \cite{Wi} studied cyclic subgroups $G$ of $\aut(C)$ and proved the following:
\begin{equation}\label{wim}
|G| \leq 4g+2.
\end{equation}

The higher dimensional counterpart to the genus is given by the plurigenera of a variety, more precisely their asymptotic behavior: the Kodaira dimension. This plays the major role in the Enriques-Kodaira classification of compact complex surfaces, in the mid-\nth{20} century, and the counterpart to Klein's theorem has been given by Aldo Andreotti in 1950: a general type surface (a surface of maximum Kodaira dimension) has finite self-bimeromorphism group. Several authors
provided polynomial bounds for automorphism groups depending on the
Chern numbers of the manifold, see for
example \cite{HS,Cor,X1}. Xiao \cite{X1} obtained a linear bound for the automorphism group of a minimal surface $X$ of general type:
\begin{equation}\label{xiao}
|\aut(X)| \leq (42K_X)^2.
\end{equation}

In the late \nth{20} century, the Kodaira dimension has been extended to foliated surfaces, i.e., a surface equipped with a foliation. Foliated surfaces have been classified, independently, by Michael McQuillan and Luis Gustavo Mendes. In 2002, Jorge Vit\'orio Pereira and Percy Fern\'andez S\'anchez \cite{PS} proved a foliated version of Andreotti's theorem: foliated surfaces of general type have finite self-bimeromorphism groups. In 2014, Maur\'icio Corr\^ea and Thiago Fassarella \cite{CF} obtained an exponential bound for foliations with ample canonical bundles and finite automorphism groups, not necessarily of general type. In \cite{BS} Sc\'ardua studied polynomial and transcendental automorphism of foliations on $\mathbb{C}^2$. 

The aim of this work is to give polynomial bounds for the order of the automorphism group $\aut(\F)$ of a foliation $\F$ in several distinct cases. We use techniques that explore the $\aut(\F)$--invariant loci to bound its order. The main idea is, after reducing to an abelian subgroup, to find appropriate invariant divisors on the surface where numerical data of the group and the foliation can be measured. 

This paper is organized as follows. In Section 3, we analyze some
local properties of automorphisms of germs of foliations. In Section
4, the automorphisms of foliations on the projective plane $\PP^2$
are studied. It is proved that the order of the automorphism group of a  foliation, degree $d$,  can be bounded quadratically with respect to $d$.
This bound is sharp, it is attained by the automorphism groups of the Jouanolou's foliations. Section 5 concerns the case when $X$ is geometrically ruled. Results toward the classification of foliations on such surfaces are proved and, under mild hypotheses, quadratic bounds are given. Finally, in the last section we study foliations with ample canonical bundle on surfaces whose Kodaira dimension is nonnegative. Cubic bounds are given for their automorphism groups under mild restrictions.

\section{Holomorphic Foliations on Surfaces}
Let $X$ be a smooth compact complex surface. A
foliation $\F$ on $X$ is given by an open covering
$\mathcal{U}=\{U_{i}\}$ of $X$ and
 $1$-forms $\omega_{i} \in
\Omega^{1}_{X}(U_{i})$ subject to the conditions:
\begin{enumerate}
 \item For each non-empty intersection $U_{i} \cap U_{j}$ there exists a holomorphic function $g_{ij} \in \OO^{*}_X(U_{i} \cap U_{j})$ such that
 $\omega_{i}= g_{ij} \omega_{j}$;
 \item For every $i$ the zero set of $\omega_{i}$ is isolated.
\end{enumerate}
The $1$-forms $\{\omega_i\}$ patch together to form a global section
$$
\omega=\{\omega_i\}\in{\rm H}^0(X,\Omega^{1}_{X}\otimes N_{\F}),
$$
where $ N_{\F}$ is the line bundle over $X$ determined by the
cocycle $\{g_{ij}\}$, which is called the normal bundle of $\F$. The singular set of $\F$, denoted by
${\rm{Sing}}(\F)$, is the zero set of the twisted $1$-form $\omega$.

Let $T_X$ be the tangent sheaf of $X$. The tangent sheaf of
 $\F$, induced by a twisted $1$-form $\omega\in
\HH^0(X,\Omega^{1}_{X}\otimes  N_{\F})$, is defined on each
open set $U\subset X$ by
$$
T_{\F}(U)=\{v\in T_X(U)\mid i_{v}\omega=0\}.
$$
The foliation $\F$ can also be given by vector fields $v_{i}\in T_X(U_{i})$ with codimension two zero set and satisfying $v_{i}=f_{ij}v_{j}$, where $f_{ij}\in \OO^{*}_X(U_{i} \cap U_{j})$. The line bundle determined by the cocycle
$\{f_{ij}\}$ is called the canonical bundle of $\F$ and it is
denoted by $K_{\F}$. 

Let $\eta\in
\HH^0(X,\Omega^{1}_{X}\otimes  N_{\F})$ be another section defining $\F$. Up to refining to a common open covering, we have that
\[
\eta_i = h_i \omega_i, \, h_i \in \OO_X^{\ast}(U_i).
\]
It follows that the functions  $h_i$ patch together to form $h \in \HH^0(X,\OO_X^{\ast})= \C^{\ast}$. Then the $h_i$ are constant. The same holds for vector fields.

We need to recall some index formulae that will be useful in  this work. For details and proofs see \cite{Bru}. Fix a smooth compact complex surface $X$ and a foliation $\F$ on $X$. Let $p\in \sing(\F)$ and let $v = A(z,w)\frac{\partial}{\partial z} + B(z,w)\frac{\partial}{\partial w}$ be a holomorphic vector field that generates $\F$ in coordinates $(z,w)$ around $p$. Then we define the residues:
\begin{align*}
{\rm BB}(\F, p) & = {\rm Res}_{(0,0)}\left\{\frac{({\rm tr}(Dv(z,w)))^2}{A(z,w)B(z,w)}\dz\wedge \dw\right\}, \\
\mu(\F, p) & = {\rm Res}_{(0,0)}\left\{\frac{({\rm det}(Dv(z,w)))^2}{A(z,w)B(z,w)}\dz\wedge \dw\right\} , 
\end{align*}
where ${\rm Res}_{(0,0)}$ denotes the Grothendieck residue at the origin. The Baum--Bott Theorem tells  us that:
\begin{align}
\sum\limits_{p\in {\rm Sing}(\F)}{\rm BB}(\F, p) 
& =  N_{\F}^2 = c_1(X)^2 -2K_{\F}\cdot K_X +K_{\F}^2,\\
\sum\limits_{p\in {\rm Sing}(\F)}\mu(\F, p)
& = c_2(T_X\otimes K_{\F}) = c_2(X) -K_{\F}\cdot K_X +K_{\F}^2 \label{BB}.
\end{align}

Let $C \subset X$ be a curve such that none of its irreducible components is invariant by $\F$. Let $\{f=0\}$ be a reduced equation for $C$ in some small neighborhood of a point $p\in C$.  Let $v$ be  a local vector field defining $\F$ around $p$, with isolated zeros.  The tangency index between $\F$ and $C$ is defined at $p$ by 
\[
\tg(\F,C,p) = \dim_{\C} \frac{\OO_{X,p}}{\left<f, v(f)\right>} .
\]
Since  $C$ is not $\F$--invariant then $v(f)$ is not identically zero along $C$. The number tangencies between $C$ and the leaves of $\F$ is finite and their sum, denoted by ${\rm tang}(\F, C)$,   can be calculated by the following:
\begin{equation}\label{tang}
\tg(\F,C) = C^2 + K_{\F}\cdot C.
\end{equation}

Let $C \subset X$ be a curve whose irreducible components are $\F$--invariant. Let $\{f=0\}$ be a reduced equation for $C$ in some small neighborhood of a point $p\in C$. If $\omega$ is a   holomorphic  $1$-form defining $\F$ around $p$, we can decompose $\omega$ (see \cite[Chapter V]{Suwa}) as follows 
\[
g\omega = h\mbox{d}f+f\eta ,
\]
where $\eta$ is a holomorphic $1$-form, $g$ and $h$ are holomorphic functions with $h$ and $f$ relatively  prime. The function $\frac{h}{g}|_C$ is meromorphic and does not depend on the choice of $g$, $h$ and $\eta$. We define
\[
{\rm Z}(\F, C, p) = ``{\rm vanishing} \, {\rm order} \, {\rm of }\, \left.\frac{h}{g}\right|_C \, \  {\rm at } \, \  p".
\]
If $p$ is a smooth point of $C$, this number  is the  Gomez-Mont--Seade--Verjovsky index regarded as the Poincaré--Hopf index of the restriction to $C$ of a local vector field defining $\F$ at $p$. This index is shown to be zero if $p$ is a regular point for $\F$, hence the sum is finite and it is calculated by the following:
\begin{equation}\label{gsv}
{\rm Z}(\F,C) = \sum\limits_{p\in {\rm Sing}(\F) \cap C}{\rm Z}(\F,C,p) = \chi(C) + K_{\F}\cdot C,
\end{equation}
where $\chi(C) = -C^2 - K_X\cdot C$ is the virtual (arithmetic) Euler characteristic of $C$.

\subsection{Kodaira Dimension of Foliations}
The notion of Kodaira dimension for holomorphic foliations has been
introduced independently by L. G. Mendes and M. McQuillan. For more
information on the subject see \cite{Bru}. For the convenience of the
reader we will recall it in the next few lines.

A singularity $x$ of $\mathcal F$ is called reduced
if the eigenvalues of the linear part $Dv_x$, of a germ of vector
field $v$ defining $\F$ at $x$, are not both zero and their
quotient, when defined, is not a positive rational number. A
foliation is called reduced if all the singularities are reduced. A
theorem due to A. Seidenberg (see \cite{Sei}) says that there exists a
sequence of blowing--ups $\pi\colon \widetilde{X} \longrightarrow X$ over the
singularities of $\F$ such that the induced foliation $\pi^{*}(\F)$
in $\widetilde{X}$ has only reduced singularities. Any reduced foliation
birationally equivalent to $\F$ is called a reduced model of
$\F$.

Let $\F$ be a foliation on the complex surface $X$ and let $\G$ be a reduced model of $\F$. The Kodaira dimension of $\F$, denoted  $\kod(\F)$, is defined as the Iitaka-Kodaira dimension of $K_\G$. Besides the birational classification of foliations  of Kodaira dimension at most one, the most important facts about general type foliations  that will be used are:
\begin{enumerate}
\item A general type foliation  $\F$ has finite bimeromorphism group, $\bim(\F)$;
\item There exists a minimal model $(Y,\G)$ à la Zariski, that is,
\[
\bim(\F) = \aut(\G).
\]
\end{enumerate} 
Hence we will focus on the automorphism groups.

\subsection{Automorphisms of Foliations}

We say that an automorphism $\varphi \colon X\longrightarrow X$ preserves a foliation $\F$ defined by a twisted $1$-form $\omega\in{\rm H}^0(X,\Omega^{1}_{X}\otimes N_{\F})$ if $\varphi^{*}\omega$ also defines the foliation $\F$. This means that $ \varphi^{*}\omega = f\omega $ for some constant $f$.  Through this work we want to analyze the order of the group ${\rm Aut}(\F)$ which is the maximum of the subgroups of $\mathrm{Aut}(X)$ that preserve $\F$.

Any automorphism of a foliation leaves invariant its singular set.
Then it is useful to know how the action on a finite set can be. We
first recall the well-known orbit-stabilizer formula.

Let $G$ be a finite group acting on a finite set $M = \{x_1, \dots, 
x_r\}$ then we can define an equivalence relation in $M$ by $x_i
\sim x_j$ if they lie in the same $G$-orbit, that is, $x_i = g(x_j)$
for some $g\in G$. Hence, $M$ splits into disjoint orbits. Suppose
that there are $s$ disjoint orbits generated by $s$ elements $x_1, \dots, x_s$,
then
\[
r = \#M = \sum_{i=1}^s \#{\rm Orb}(x_i) =\sum_{i=1}^s \left(G:H_i\right) = \sum_{i=1}^s \frac{|G|}{|H_i|},
\]
where $H_i$ is the stabilizer of $x_i$ in $G$. A group of automorphisms which preserves a foliation $\F$ acts on its singular set $\sing(\F)$ and  preserves the
analytical invariants of the singularities. Concerning the Milnor numbers of the singularities, we have the following:

\begin{proposition}\label{orb}
Let $\F$ be a foliation on a compact complex surface $X$ and let
$G<\aut(\F)$ be a finite subgroup. Then
\[
c_2(T_X\otimes K_{\F}) = \sum_{i=1}^s \frac{|G|}{|H_i|}\mu(\F,x_i),
\]
where $x_1, \dots, x_s\in \sing(\F)$ lie on disjoint orbits and
$H_i$ is the stabilizer of $x_i$ in $G$.
\end{proposition}

\begin{proof}
Let $e$ be the number of singularities counted without multiplicity, then
\[
e = \sum_{i=1}^s \#{\rm Orb}(x_i) = \sum_{i=1}^s \frac{|G|}{|H_i|},
\]
where the $x_1, \dots, x_s$ generate the disjoint orbits and $H_i$
is the stabilizer of $x_i$. Let $v$ be a germ of vector field that
defines $\F$ at a singularity $p$ and let $g\in G$. Then $g_{\ast}v$
defines $\F$ at $g(p)$ and
\[
\mu(\F,p) = \mu(\F,g(p)),
\]
since the Milnor number is invariant by biholomorphisms. This
implies that all singularities in a same orbit have the same
Milnor number, hence
\[
c_2(T_X\otimes K_{\F}) = \sum_{p\in\sing(\F)} \mu(\F,p) = \sum_{i=1}^s \#{\rm
Orb}(x_i)\mu(\F,x_i) = \sum_{i=1}^s \frac{|G|}{|H_i|}\mu(\F,x_i).
\]
\end{proof}

Recall that the Milnor number is always an integer. Then we can get interesting numeric relations between $\F$ and $\aut(\F)$. The indices of a foliation along a curve are also preserved under biholomorphisms. The tangency and the ${\rm Z}$-index are interesting to us because they are also integers. 

\begin{proposition}\label{orbind}
Let $\F$ be a foliation on a compact complex surface $X$ and let
$G<\aut(\F)$ be a finite subgroup. Let $C\subset X$ be a compact $G$--invariant curve. Then:
\begin{enumerate}
\item If all irreducible components of $C$ are not $\F$--invariant, then
\[
\tg(\F,C) = \sum_{i=1}^s \frac{|G|}{|H_i|}\tg(\F,C,x_i).
\] 
\item If all irreducible components of $C$ are $\F$--invariant, then
\[
{\rm Z}(\F,C) = \sum_{i=1}^s \frac{|G|}{|H_i|}{\rm Z}(\F,C,x_i).
\]
\end{enumerate}
The points $x_1, \dots, x_s\in C$ lie on disjoint orbits and $H_i$ is the stabilizer of $x_i$ in $G$.
\end{proposition}

\begin{proof}
Since $C$ is $G$--invariant and the indices are invariant under biholomorphisms, the proof 
follows, mutatis mutandis, the argument of Proposition \ref{orb}.
\end{proof}

\section{Automorphisms of Germs of Foliations}

This section is devoted to establish technical results about finite groups acting on $(\C^2,0)$ that preserve a given germ of foliation. The main purpose of this local study is to be able to understand, in the global case, the behavior of the automorphism group of a foliation around a fixed point. 

Let $\F$ be a germ of foliation on $(\C^2,0)$ defined by a germ of vector field $v$. If $w$ is another germ of vector field defining $\F$, then there exists a germ of nonvanishing holomorphic function $f\in \OO_{\C^2,0}^{\ast}$ such that
\[
w = fv.
\]
Therefore, an automorphism $\varphi \in {\rm Diff}(\C^2,0)$ preserves $\F$ if and only if $\varphi_{\ast}v = fv$ for some $f\in \OO_{\C^2,0}^{\ast}$. We define the group 
\[
\aut(\F,0) \ceq \left\{\varphi \in {\rm Diff}(\C^2,0) \mid \exists f\in \OO_{\C^2,0}^{\ast} \mid  \varphi_{\ast}v = fv \right\},
\]
of automorphisms that preserve $\F$. Since we are interested only in the local behavior of automorphisms of compact surfaces around a fixed point, we make the further assumption that, for some choice of $v$, $f$ is constant.  We will consider the finite subgroups of the following group:
\[
\aut(v,0) \ceq \left\{\varphi \in {\rm Diff}(\C^2,0) \mid \exists f\in \C^{\ast} \mid  \varphi_{\ast}v = fv \right\} < \aut(\F,0).
\]
Observe that one can state an analogous definition with germs of $1$-forms. 

The first two general facts to notice are that every finite group of germs of holomorphic automorphisms is linearizable and well-behaved under blow-ups.

\begin{lemma}\label{finlin}
Let $G$ be a finite subgroup of ${\rm Diff}(\C^n,0)$. There exists $T\in {\rm Diff}(\C^n,0)$ such that $TGT^{-1}\subset \gl(n,\C)$.
\end{lemma}

\begin{proof}
Define $T$ by
\[
T\ceq \frac{1}{|G|} \sum_{\varphi\in G} D\varphi_0^{-1}\cdot\varphi.
\]
It is clear that $DT_0$ is the identity, hence $T\in {\rm Diff}(\C^n,0)$. Also, observe that $T\circ \varphi = D\varphi_0\cdot T$ for every $\varphi \in G$. Then 
\[
T\varphi T^{-1} =  D\varphi_0 \in \gl(n,\C).
\]
\end{proof}

\begin{remark}\label{cic}
Notice that for $n=1$ the group $G$ is isomorphic to a subgroup of $\gl(1,\C) = \C^{\ast}$, hence it is cyclic. In particular, a finite group acting faithfully on a curve with a smooth fixed point is cyclic.
\end{remark}

Consider   $\pi\colon (X,E) \longrightarrow (\C^2,0)$   the blow-up at the origin, where $E$ denotes the exceptional divisor. Then the following holds.

\begin{lemma}\label{blowup}
Let $G$ be a finite subgroup of ${\rm Diff}(\C^2,0)$. Then, $G$ acts faithfully on $(X,E)$.
\end{lemma}

\begin{proof}
Fix an element of $G$, say 
\[
\varphi(x,y) = (ax+by, cx+dy), 
\]
with $ad-bc \neq 0$. Choosing coordinates $(x,y;u: t)$ on $(X,E)$ such that
$\pi(x,y;u: t) = (x,y)$, we can define an automorphism of $(X,E)$ by
\[
\tilde{\varphi}(x,y;u : t) = (ax+by, cx+dy; au+bt : cu+dt).
\]
The map $\varphi \mapsto \tilde{\varphi}$ defines a homomorphism with the property $\varphi\circ\pi = \pi\circ\tilde{\varphi}$. Since $\pi$ is a germ of biholomorphism outside the exceptional divisor, this map is injective.
\end{proof} 

Before we proceed to the study of the automorphisms of germs of foliations, we need to state a simple technical lemma.

\begin{lemma}\label{cong}
Let $G$ be a finite linear diagonal subgroup of ${\rm Diff}(\C^2,0)$. A general element of $G$ is written as
\[
\varphi(x,y)=(l^ax,l^by)
\] 
with $l^n =1$ primitive, $a,b\in \Z$ and $\gcd(a,b,n)=1$. If there exists $u\in\Z$ such that
\[
a\equiv bu\pmod{n} 
\]
for all $\varphi\in G$, then $G$ is cyclic. Moreover, $G$ has a generator of the form
\[
\psi(x,y) = (\zeta^{u}x, \zeta y),
\]
where $\zeta$ is a root of the unity.
\end{lemma}

\begin{proof}
The equation $a\equiv bu\pmod{n}$ implies that $\gcd(b,n)\mid a$. Since $\gcd(a,b,n)=1$ we have that $\gcd(b,n)=1$. Let $t\in \Z$ such that $bt\equiv 1 \pmod{n}$, then
\[
\varphi^t(x,y) = (l^{u}x,ly),
\]
which generates the same group as $\varphi$. 

Suppose that $n$ is the maximum order of an element of $G$. We need to show that $\varphi$ generates $G$. Let $\psi$ be another element in $G$. Then we can assume that 
\[
\psi(x,y) = (\zeta^{u}x, \zeta y),
\]
with $\zeta^m = 1$ primitive and $m\leq n$. The group $G$ must have an element of order $\lcm(m,n) \geq n$. By the maximality of $n$, we have that $n=md$ for some $d\geq1$. Therefore, $\zeta = l^{sd}$ for some $s$ invertible modulo $n$, which implies $\psi \in \left<\varphi\right>$. Hence $G$ is cyclic.
\end{proof}

\subsection{Automorphisms of Regular Foliations}

Let $\F$ be a germ of regular foliation on $(\C^2,0)$ given by a germ vector field $v$. Denote by $\omega$ the dual $1$-form. By regularity, we may assume that
\[
\omega = \di f,
\]
where $f$ is a germ of submersion. For any $\varphi$ that preserves $\F$, there exists an automorphism $\chi(\varphi)\in {\rm Diff}(\C,0)$ such that 
\[
f \circ \varphi^{-1} = \chi(\varphi) \circ f.
\]
Notice that for another $\psi\in\aut(\F,0)$, 
\[
\chi(\varphi\circ \psi^{-1}) \circ f = f \circ  \psi\circ\varphi^{-1} = (\chi(\psi^{-1}) \circ f)\circ\varphi^{-1} = \chi(\psi^{-1})\circ\chi(\varphi)\circ f.
\]
Then the map $\chi\colon \aut(\F,0) \longrightarrow {\rm Diff}(\C,0)$ is an antihomomorphism. We will see that, for any $G<\aut(\F,0)$ finite, $G$, $\chi(G)$ and $\F$ can be simultaneously linearized.

\begin{proposition}\label{streg}
Let $\F$ be a germ of regular foliation on $(\C^2,0)$ and let $G$ be a
finite subgroup of $\aut(\F,0)$. Then $G$ is abelian. Moreover, there exist coordinates $(x,y)$ where $G$ is linear diagonal and $\F$ is given by $\mbox{dx}$ or, equivalently, by $\partial_y$.
\end{proposition}

\begin{proof}
Let $f$ be a holomorphic submersion defining $\F$. By Lemma \ref{finlin}, there exist $T \in {\rm Diff}(\C^2,0)$ and $U\in {\rm Diff}(\C,0)$ such that $TGT^{-1}$ and $U\chi(G)U^{-1}$ are linear. The restriction of $\chi$ to $G$ is an homomorphism by Remark \ref{cic}.  In these new coordinates, $\F$ is given by $U\circ f\circ T^{-1}$. Then we may assume that $G$ is linear and
\begin{equation}\label{lin1}
f\circ \varphi^{-1} = \chi(\varphi) f, \, \chi(\varphi)\in \C^{\ast},
\end{equation}
for every $\varphi \in G$. Suppose without loss of generality that $f$ is tangent to the $y$-axis ($\frac{\partial f}{\partial y}(0) = 0$) and let 
\[
\varphi(x,y)=(ax+by,cx+dy)
\] 
be an element of $G$. The equation (\ref{lin1}) says that the tangent vector to $f$ at the origin is an eigenvector for $\varphi$. Then $b=0$ and $G$ is a group of lower triangular matrices. Every finite group of lower triangular matrices is abelian. Indeed, the commutator of two such matrices either has infinite order or is the identity. Up to a linear change of coordinates, $G$ is diagonal and $f$ is still tangent to the $y$-axis.

Now, let $S\colon  (\C^2,0) \longrightarrow (\C^2,0)$ be defined by
\[
S(x,y) = (f(x,y), y).
\]
Observe  that $S\in {\rm Diff}(\C^2,0)$ and $f\circ S^{-1}=x$. For $\varphi(x,y)=(ax,dy)\in G$ we have that 
\[
S\circ\varphi\circ S^{-1}(x,y) = \left(\chi(\varphi^{-1})(f\circ S^{-1})(x,y),d(y\circ S^{-1})(x,y)\right) = \left(\chi(\varphi^{-1})x,dy\right).
\]

Therefore, after changing the coordinates,  $G$ is linear diagonal and $\F$ is given by $\dx$ or, equivalently, $\partial_y$.
\end{proof}

\begin{corollary}\label{fibloc}
Let $\F$ be a germ of regular foliation on $(\C^2,0)$ and let $G$ be a
finite subgroup of $\aut(\F,0)$. Then there are cyclic groups $K$ and $H$ such that the sequence
\[
1 \longrightarrow K \longrightarrow G \longrightarrow H \longrightarrow 1
\]
is exact. The group $K$ maps each leaf of $\F$ onto itself.
\end{corollary}

\begin{proof}
By Proposition \ref{streg}, we can choose coordinates coordinates where $G$ is abelian and $\F$ is defined by $\dx$. Moreover, 
\[
\chi\colon G \longrightarrow {\rm Diff}(\C,0)
\]
is an homomorphism and any element $\varphi \in G$ can be written as 
\[
\varphi(x,y) = \left(\chi(\varphi^{-1})x,dy\right).
\]
Observe that the leaves of $\F$ are the curves $\{x=c\}$, with $c$ constant. To conclude, define
\[
K\ceq \ker(\chi), \, H\ceq \chi(G).
\]
Remark \ref{cic} implies that both $K$ and $H$ are cyclic.
\end{proof}

We have seen that there are coordinates where a foliation $\F$ and a finite group of automorphisms $G$ can be both linearized. In the presence of a $G$--invarant smooth curve that is not $\F$--invariant, we can show that $G$ must be cyclic.

\begin{proposition}\label{tgreg}
Let $\F$ be a germ of regular foliation on $(\C^2,0)$ and let $G$ be a finite subgroup of $\aut(\F,0)$. Let $C$ be a $G$--invariant germ of smooth curve not $\F$--invariant. If $\tg(\F,C, 0)\geq 1$ then, $G$ is cyclic and  there exist coordinates $(x,y)$ such that $G$ is generated by an element of the form
\[
\varphi(x,y) = (l^{k+1}x,ly),
\]
where $l$ is a root of the unity and $k = \tg(\F,C, 0)$. 
\end{proposition}

\begin{proof}
By Proposition \ref{streg}, we may take coordinates $(x,y)$ such that $G$ is diagonal and $\F$ is given by $\partial_y$. Let $\{g=0\}$ be a reduced equation for $C$. Then $x$ does not divide $g$, since $C$ is not $\F$--invariant. By the Weierstrass Preparation Theorem, we may assume that $g = x + y^{k+1}h(y)$, where $k\geq0$ and $h$ is holomorphic and $h(0)\neq 0$. The second nontrivial jet of $g$ at the origin is 
\[
j_0^{k+1}g(x,y) = x + y^{k+1}h(0).
\]
Let $\varphi(x,y)=(l^ax,l^by)$ be an element of $G$ whose order is $n$ ($l^n=1$  primitive and $\gcd(a,b,n)=1$). Then
\[
j_0^{k+1}g\circ\varphi(x,y) = l^ax + l^{(k+1)b}y^{k+1}h(0), 
\]
which implies that $a\equiv (k+1)b \pmod{n}$. By Lemma \ref{cong}, $G$ is cyclic and we can assume that 
\[
\varphi(x,y) = (l^{k+1}x,ly)
\]
generates $G$. It remains to show that $k=\tg(\F,C,0)$. Observe that we have $\frac{\partial g}{\partial y}=y^ku(y)$ with $u(0)\neq 0$. Then, 
\[
\tg(\F,C,0) = {\rm dim}_{\C}\frac{\mathcal{O}_{\C^2,0}}{\left< g, \partial_y(g)\right>} = {\rm dim}_{\C}\frac{\mathcal{O}_{\C^2,0}}{\left< x+y^{k+1}h(y), y^k \right>} = {\rm dim}_{\C}\frac{\mathcal{O}_{\C^2,0}}{\left< x, y^k \right>} = k.
\]
\end{proof}

\begin{proposition}\label{tgvec}
Let $\F$ be a germ of regular foliation on $(\C^2,0)$ and let $v$ be a germ of vector field with an isolated singularity at the origin. Let $G$ be a finite subgroup of $\aut(\F,0)\cap \aut(v,0)$. Denote by $\F_0$ the germ of $\F$-leaf at the origin. If $v$ does not leave $\F_0$ invariant, then the subgroup $K<G$ which fixes the leaves of $\F$ has bounded order:
\[
|K| \leq \tg(v,\F_0,0)+1.
\]
\end{proposition}

\begin{proof}
By Proposition \ref{streg}, we may choose coordinates $(x,y)$ where $G$ is diagonal and $\F$ is defined by $\dx$. As in Corollary \ref{fibloc}, in these coordinates $K$ is generated by 
\[
\varphi(x,y)=(x,ly),
\]
with $l^{|K|}=1$ primitive. Let 
\[
v = A(x,y)\partial_x +B(x,y)\partial_y
\]
be the expression of $v$ in these coordinates. Also observe  that $\F_0=\{x=0\}$, hence $x$ does not divide $A$. We have two cases: whether $y$ divides $A$ or not.

Suppose that $y$ does not divide $A$. The Weierstrass Preparation Theorem implies that
\[
A(x,y) = u(x,y)\left(a_0(y)+a_1(y)x + \dots + a_{k-1}(y)x^{k-1} +x^k \right),
\]
for some holomorphic functions $u$ and $a_i$, $i=0,\dots k-1$ with $u(0)\neq0$, $k>0$. Since $v$ has a singularity at the origin, $a_0(y) = y^rb(y)$, $b(0)\neq 0$ and $r>0$. The pushforward of $v$ by $\varphi$ is
\[
\varphi_{\ast} v = A\circ\varphi^{-1}\partial_x + lB\circ\varphi^{-1}\partial_y .
\]
Hence $l^r=1$ (since $k>0$) or, equivalently, 
\[
r\equiv 0 \pmod{|K|}.
\]

Now suppose that $y$ divides $A$. Then $A(x,y) = u(x,y)y^r$ with $r>0$, $u$ holomorphic and $u(0)\neq0$. Since the singularity is isolated, $y$ does not divide $B$. By the Weierstrass Preparation Theorem,
\[
B(x,y) = q(x,y)\left(b_0(x)+b_1(x)y + \dots + b_{t-1}(x)y^{t-1} +y^t \right),
\]
for some holomorphic functions $q$ and $b_i$, $i=0,\dots t-1$ with $q(0)\neq0$. We have that $b_0(0) =0$, then some power of $x$ divides $b_0(x)$. It follows by taking the pushforward of $v$ by $\varphi$ that $l^{-r} = l$ or, equivalently,
\[
r\equiv -1 \pmod{|K|}.
\]

Either way, we have 
\[
|K|\leq r+1.
\]
It remains to show that $r$ is the tangency index:
\[
\tg(v,\F_0,0) = {\rm dim}_{\C}\frac{\mathcal{O}_{\C^2,0}}{\left< x , v(x)\right>} = {\rm dim}_{\C}\frac{\mathcal{O}_{\C^2,0}}{\left< x, A \right>} = {\rm dim}_{\C}\frac{\mathcal{O}_{\C^2,0}}{\left< x, y^r \right>} = r.
\]
\end{proof}

\subsection{Automorphisms of Singular Foliations}

Naturally, the behavior of the foliation near its singular points
also imposes many restrictions on its automorphism group. The finite groups that preserve reduced singularities are described in the
following proposition:

\begin{proposition}\label{stab}
Let $\F$ be a germ of singular foliation on $(\C^2,0)$ with reduced singularity whose eigenvalue is $\lambda$. Let $v$ be a vector field defining $\F$ and let $G$ be a finite subgroup of $\aut(v,0)$. Then there exist coordinates such that 
\[
v = (x+P(x,y))\partial_x + (\lambda y + Q(x,y))\partial_y,
\]
where $P$ and $Q$ vanish at the origin with order at least two, and one of the following holds:
\begin{enumerate}
\item $\lambda \neq -1$ and $G$ is abelian diagonal.
\item $\lambda = -1$ and $G$ has an abelian diagonal subgroup $H$ whose index is at most two.
\end{enumerate}
In the second case, if $G$ is abelian but not diagonal ($G\neq H$),  then $G$ has generators
\[
\varphi(x,y) = (lx,ly) \, {\rm and } \, \psi(x,y) = (\zeta y,\zeta x), 
\] 
where $l^n=1$, $\zeta^m=1$ and $\zeta^2=l^s$, for some integers satisfying $2n=sm$. 
\end{proposition}

\begin{proof}
According to Lemma \ref{finlin}, we may choose coordinates $(x,y)$ such that $G$ is linear. Fix an element $\varphi$ of $G$ written as
\[
\varphi (x,y) = (ax+by, cx+dy),
\]
where $ad-bc \neq 0$. By hypothesis, $Dv_0$ is not zero and the relation $\varphi_{\ast}v = fv$, $f\neq 0$, implies that
\begin{equation}\label{comm}
f Dv_0 = D\varphi_0 \cdot Dv_0 \cdot (D\varphi_0)^{-1},
\end{equation}
since $v(0) = 0$. Up to a linear change of coordinates, we may assume that $Dv_0$ is in its canonical form:
\[
\left(\begin{array}{cc}
1 & 0 \\
0 & \lambda
\end{array}\right),
\]
where $\lambda \notin \Q_{>0} $. Hence, $v$ is written as
\[
v = (x+P(x,y))\partial_x + (\lambda y + Q(x,y))\partial_y.
\]
Therefore, the equation (\ref{comm}) becomes
\[
f(ad-bc)\left(\begin{array}{cc}
1 & 0 \\
0 & \lambda
\end{array}\right) =
\left(\begin{array}{cc}
ad-\lambda bc & ab(\lambda -1) \\
cd(1-\lambda) & ad\lambda - bc
\end{array}\right).
\]
Since $\lambda \neq 1$ and $ad-bc \neq 0$, we have either $b=c=0$ or $a=d=0$. If $b=c=0$, then
\[
\left\{\begin{array}{c}
fad = ad \\
fad\lambda = ad\lambda
\end{array}\right. .
\]
Hence, $f = 1$ and 
\[
\varphi (x,y) = (ax, dy),
\] 
with $a,d \in \C^{\ast}$ roots of the unity. If $a=d=0$, then
\[
\left\{\begin{array}{c}
-fbc = -\lambda bc \\
-fbc\lambda = -bc
\end{array}\right. ,
\]
which implies that $f = \lambda = -1$ and 
\[
\varphi (x,y) = (by, cx),
\]
with $b,c\in\C^{\ast}$ roots of the unity. Consequently, if $\lambda \neq -1$ then $G$ is diagonal, proving the first case. Now, we show that, when $\lambda=-1$, $G$ has a diagonal subgroup of index at most two. 

Let $H$ be the subgroup of $G$ generated by all diagonal elements. We will show that $\left(G:H\right) \leq 2$. Suppose that $H\neq G$. Then $G$ contains an element of the form 
\[
\psi(x,y) = (by,cx).
\]
For any such element, $\psi^2\in H$. If $\rho$ is another element in $G \setminus H$, then $\rho\circ\psi \in H$. Hence $\rho \equiv \psi \pmod{H}$, which implies that 
\[
\left(G:H\right)= 2.
\]

To finish, suppose that $G$ is abelian but not diagonal, $G\neq H$. Then, let 
\[
\psi(x,y) = (by,cx)
\]
be an element of $G\setminus H$. Up to conjugating $G$ with 
\[
T(x,y) = \left(\sqrt{\frac{c}{b}}x,y\right),
\]
we may assume that $b=c$. Indeed, another element $\rho(x,y) = (py,qx)$ commutes with $\psi$ if and only if $cp = qb$. Calculating the commutator with $\psi$ we see that every element of $H$ has the form
\[
\varphi(x,y) = (lx,ly),
\]
where $l$ is a root of the unity. It follows that $H$ is isomorphic to a subgroup of $\C^{\ast}$, hence it is cyclic. Therefore, $G$ has generators
\[
\varphi(x,y) = (lx,ly) \, {\rm and } \, \psi(x,y) = (by,bx), 
\] 
where $l^n=1$, $b^m=1$ and $b^2=l^s$, for some $s$ satisfying $2n=sm$.
\end{proof}

\begin{proposition}\label{fibtg}
Let $\F$ be a germ of singular foliation on $(\C^2,0)$ defined by a germ vector field $v$ with reduced singularity whose eigenvalue is $\lambda$. Let $G$ be a finite subgroup of $\aut(v,0)$ and let $C$ be a germ of smooth $G$--invariant curve in $(\C^2,0)$ which is not $\F$--invariant. Then there exist coordinates such that $C$ is  tangent to the $y$-axis and $G$ satisfies one of the following:
\begin{enumerate}
\item $G$ is cyclic, generated by 
\[
\varphi(x,y) = (l^kx,ly),
\]
where $l$ is a root of the unity and $k = \tg(\F,C, 0) \geq 1$;
\item $G$ is abelian but not cyclic with generators
\[
\varphi(x,y) = (lx,ly) \, {\rm and } \, \psi(x,y) = (-\zeta x,\zeta y)\, {\rm or }\,(\zeta x,-\zeta y), 
\] 
where $l^n=1$, $\zeta^m=1$ and $\zeta^2=l^s$, for some integers satisfying $2n=sm$.
\end{enumerate}
In the second case, $\lambda =-1$ and $\tg(\F,C, 0)=1$.
\end{proposition}

\begin{proof}
Let $\{g=0\}$ be an equation for $C$. For any $\varphi \in G$ there exists $\sigma(\varphi)\in \OO_{\C,0}^{\ast}$ such that 
\[
g \circ\varphi = \sigma(\varphi)g.
\]
Then we can construct a $G$--invariant regular foliation $\G$. Indeed, let $\G$ be given by the $1$-form
\[
\eta = \frac{1}{|G|}\sum_{\varphi\in G}\frac{1}{\chi(\varphi)}\varphi^{\ast}\di g.
\]
It follows that $\varphi^{\ast}\eta = \chi(\varphi)\eta$, hence $G< \aut(\G,0)$. Proposition \ref{streg} implies that $G$ is abelian. By Proposition \ref{stab}, we can choose coordinates $(x,y)$ such that
\[
v = (x+P(x,y))\partial_x + (\lambda y + Q(x,y))\partial_y,
\]
where $P$ and $Q$ are germs of holomorphic functions that vanish at the origin with order at least two and $G$ falls in one of the following two cases:
\begin{enumerate}
\item  $G$ is diagonal;
\item $G$ is abelian but not diagonal with generators
\[
\varphi(x,y) = (lx,ly) \, {\rm and } \, \psi(x,y) = (\zeta y,\zeta x), 
\] 
where $l^n=1$, $\zeta^m=1$ and $\zeta^2=l^s$, for some integers satisfying $2n=sm$. 
\end{enumerate}
Observe  that the second case only occurs if $\lambda=-1$.

First case: $G$ is diagonal. \\
Since the  curve $C$ is smooth and $G$--invariant, the tangent vector to $C$ at the origin is an eigenvector to every element of $G$. We have to analyze the relative position between this vector and the coordinate axes.

Suppose that $C$ is transverse to both axes. Then every element of $G$ has a two dimensional eigenspace, hence $G$ is cyclic and  generated by 
\[
\varphi(x,y) = (lx, ly).
\]
where $l$ is a root of the unity. If the linear part of $g$ is $x+uy$, then 
\[
\left.v(g)\right|_{g=0} = (1-\lambda)uy + h(y)
\]
where $h$ vanishes at the origin with order at least two. Hence,
\[
\tg(\F,C, 0) = {\rm dim}_{\C}\frac{\mathcal{O}_{\C^2,0}}{\left< g, v(g)\right>} = {\rm dim}_{\C}\frac{\mathcal{O}_{\C,0}}{\left< \left.v(g)\right|_{g=0} \right>} = 1.
\]

Now suppose that $C$ is tangent to one of the axes. Let 
\[
\varphi(x,y) = (l^ax, l^by)
\]
be an element of $G$ with $l^n=1$ primitive and $\gcd(a,b,n)=1$. Assume that $C$ is tangent to the $y$-axis, $\{x=0\}$. If $x$ does not divide $g$, then the Weierstrass Preparation Theorem implies that 
\[
g(x,y) = u(x,y)(x - y^rh(y)),
\]
where $u$ and $h$ are germs of non-vanishing functions, and $r>1$. Composing $g$ with $\varphi$ we have that $l^a = l^{rb}$, hence
\[
a \equiv rb \pmod{n}.
\]
Lemma \ref{cong} implies that $G$ is cyclic and has a generator of the form
\[
\psi(x,y) = (l^{r}x,ly).
\] 
In this case we have that $\left.v(g)\right|_{g=0} = y^r[u(0)h(0)(1-\lambda r)+\dots]$, hence
\[
\tg(\F,C, 0) = {\rm dim}_{\C}\frac{\mathcal{O}_{\C^2,0}}{\left< g, v(g)\right>} = {\rm dim}_{\C}\frac{\mathcal{O}_{\C,0}}{\left< \left.v(g)\right|_{g=0} \right>} = r.
\]
If $x$ divides $g$ ($g=u(x,y)x$ with $u(0)\neq 0$), then $x$ cannot divide $x+P$ since $C$ is not $\F$--invariant. By the Weierstrass Preparation Theorem, 
\[
x + P = c(x,y)(x-y^ks(y)),
\]
where $c$ and $s$ are germs of non-vanishing functions and $k>1$. Taking the pushforward of $v$ by $\varphi$ we have that $l^a = l^{kb}$, hence
\[
a \equiv kb \pmod{n}.
\]
Again by Lemma \ref{cong}, $G$ is cyclic and has a generator of the form
\[
\psi(x,y) = (l^{k}x,ly).
\]
We have that 
\[
\tg(\F,C, 0) = {\rm dim}_{\C}\frac{\mathcal{O}_{\C^2,0}}{\left< g, v(g)\right>} = {\rm dim}_{\C}\frac{\mathcal{O}_{\C^2,0}}{\left< x, y^k\right>} = k.
\]

Now assume that $C$ is tangent to $\{y=0\}$. If $\lambda\neq0$, then the argument is the same as above. Indeed, we may divide $v$ by $\lambda$ and exchange $x$ with $y$. Hence we have to consider only the case where $\lambda=0$. If $y$ does not divide $g$, then the Weierstrass Preparation Theorem implies that 
\[
g(x,y) = u(x,y)(y - x^rh(x)),
\]
where $u$ and $h$ are germs of non-vanishing functions, and $r>1$. Composing $g$ with $\varphi$ we have that $l^{ra} = l^b$, hence
\[
ar \equiv b \pmod{n}.
\]
Lemma \ref{cong} implies that $G$ is cyclic and has a generator of the form
\[
\psi(x,y) = (lx,l^{r}y).
\] 
In this case we have that $\left.v(g)\right|_{g=0} = x^r[u(0)r+\dots]$, hence
\[
\tg(\F,C, 0) = {\rm dim}_{\C}\frac{\mathcal{O}_{\C^2,0}}{\left< g, v(g)\right>} = {\rm dim}_{\C}\frac{\mathcal{O}_{\C,0}}{\left< \left.v(g)\right|_{g=0} \right>} = r.
\]
If $y$ divides $g$ ($g=u(x,y)y$ with $u(0)\neq 0$) then $y$ cannot divide $Q$. By the Weierstrass Preparation Theorem we have
\[
Q= p(x,y)(x^m - q(x,y)),
\]
where $p$ and $q$ are germs or holomorphic functions, $m>1$, $p(0)\neq0$ and $q$ is a polynomial in $x$ that vanishes at the origin. The relation $\varphi_{\ast}v = fv$ implies that
\[
f[(x+P)\partial_x + Q(x,y)\partial_y] = (l^{-a}x+P\circ\varphi^{-1})l^a\partial_x + p\circ\varphi^{-1}( l^{-ma} x^m - q\circ\varphi^{-1})l^b\partial_y.
\]
Hence $ma \equiv b \pmod{n}$. By Lemma \ref{cong}, $G$ is cyclic and generated by 
\[
\psi(x,y) = (lx,l^my).
\] 
We have that
\[
\left.v(g)\right|_{\{g=0\}} = x^m ( p(0)u(0)+\dots).
\]
hence $\tg(\F,C, 0)=m$. 

Summarizing, for every subcase above we have coordinates $(x,y)$ where $G$ is cyclic, generated by 
\[
\varphi(x,y) = (l^kx,ly),
\]
where $l$ is a root of the unity and $k = \tg(\F,C, 0)\geq 1$. This proves our first case.

Second case: $G$ is abelian but not diagonal with generators
\[
\varphi(x,y) = (lx,ly) \, {\rm and } \, \psi(x,y) = (\zeta y,\zeta x), 
\] 
where $l^n=1$, $\zeta^m=1$ and $\zeta^2=l^s$, for some integers satisfying $2n=sm$. \\
Observe that in this case $\lambda = -1$. Since $C$ is $\psi$--invariant, $g$ is tangent to $\{x+y=0\}$ or $\{x-y=0\}$. Hence 
\[
\tg(\F,C,0)=1.
\] 
After a linear change of coordinates we have that $C$ is tangent to $\{x=0\}$ and $G$ is generated by $\varphi$ and 
\[
 \psi(x,y) = (\zeta x,-\zeta y), \, {\rm or } \,  \psi(x,y) = (-\zeta x,\zeta y).
\]
This concludes our second case.
\end{proof}

\section{Bounds for Foliations on the Projective Plane}

Foliations on the projective plane $\PP^2$ can be described by means of the Euler sequence
\[
0 \longrightarrow \OO_{\PP^2}(-1) \longrightarrow
\OO_{\PP^2}^{\oplus 3} \longrightarrow T\PP^2 \otimes
\OO_{\PP^2}(-1) \longrightarrow 0
\]
induced by the natural inclusion of the tautological line bundle
$\opn(-1)$ in the trivial bundle of rank $3$. Dualizing this
sequence we get
\[
0 \longrightarrow \Omega_{\PP^2}^1\otimes \OO_{\PP^2}(1)
\longrightarrow \OO_{\PP^2}^{\oplus 3} \longrightarrow
\OO_{\PP^2}(1) \longrightarrow 0, 
\]
where the last map is the contraction with the radial vector field
\[
R \ceq X\partial_X + Y\partial_Y + Z\partial_Z. 
\]

Let $\F$ be a foliation on $\PP^2$ defined by a global section of
$T\PP^2 \otimes \OO_{\PP^2}(k)$, for some $k\in \mathbb{Z}$. Then
$K_{\F} = \OO_{\PP^2}(k)$. The degree $d$ of $\F$ is defined by the
tangency number that $\F$ has with a general line $L$. That is, 
\[
d =\tg(\F,L) = K_{\F}\cdot L + L^2 = k+1.
\]
Therefore, a degree $d$ foliation $\F$ has canonical bundle isomorphic
to $\OO_{\PP^2}(d-1)$. The normal bundle is determined by
\[
N_{\F} = K_{\F}\otimes K_{\PP^2}^{\vee} = \OO_{\PP^2}(d+2).
\]
Tensorizing the Euler sequence by  $\OO_{\PP^2}(d)$ and taking long exact sequence of cohomology, we obtain 
\[
0 \longrightarrow \HH^0(\PP^2, \OO_{\PP^2}(d-1) ) \longrightarrow
\HH^0(\PP^2, \OO_{\PP^2}(d))^{\oplus 3} \longrightarrow
\HH^0(\PP^2, T\PP^2 \otimes \OO_{\PP^2}(d-1) ) \longrightarrow 0
\]
since $\HH^1(\PP^2, \OO_{\PP^2}(d-1) )=0$. This implies that a
holomorphic global section of $T\PP^2 \otimes \OO_{\PP^2}(d-1)$ can be 
identified with a polynomial vector field in $\C^3$
\[
v = A\partial_X + B\partial_Y + C\partial_Z,
\]
where $A$, $B$ and $C$ are homogeneous of degree $d$, and $v+ GR$
defines the same section for any homogeneous polynomial $G$ of
degree $d-1$.

Recall that the divergent of a vector field $v$ is defined in this
case by 
\[
{\rm div}(v) = \pder{A}{X}  + \pder{B}{Y} + \pder{C}{Z} = \mbox{d}(\iota_v\mu),
\]
where $\mu = \dX\wedge \dY \wedge \dZ$ and  $\iota_v$ is the contraction
with $v$ and $d$ is the differential operator. For $v+ GR$, a direct
calculation shows that
\[
{\rm div }(v + GR) = {\rm div }(v) + (d+2)G.
\]
Hence, any foliation can be represented by a vector field
with null divergent, just put $G = -{\rm div }(v)/(d+2)$. The
interaction of automorphisms and the divergent vector fields in
$\C^3$ leads to an interesting fact.

\begin{lemma}\label{div}
Let $\F$ be a foliation on $\PP^2$ given by a vector field $v$. If $F$ is an element of $\aut(\F) < \aut(\PP^2)$, then
\[
{\rm div}(F^{-1}_{\ast}v) = F^{\ast}{\rm div}(v).
\]
Moreover, if ${\rm div}(v)=0$ then $F_{\ast}v = \lambda v$ for some $\lambda \in \C^{\ast}$.
\end{lemma}

\begin{proof}
 Since  $F^{\ast}\mu = {\rm det}(F)\mu$, we conclude that 
\begin{align*}
{\rm div}(F^{-1}_{\ast}v)\mu & = d(\iota_{F^{-1}_{\ast}v}\mu) =\frac{1}{{\rm det}(F)} d(\iota_{F^{-1}_{\ast}v}F^{\ast}\mu)= \frac{1}{{\rm det}(F)} d(F^{\ast}(\iota_v\mu)) \\
& = \frac{1}{{\rm det}(F)}F^{\ast}({\rm div}(v)\mu) = (F^{\ast}{\rm div}(v))\mu .
\end{align*}
We have that $F\in \aut(\F)$ means that $F_{\ast}v$ defines $\F$. Then $F_{\ast}v = \lambda v + GR$ with $\lambda \in \C^{\ast}$ and $G$ a homogeneous polynomial. If ${\rm div}(v) = 0$, then
\[
0 = {\rm div }(F_{\ast}v) = {\rm div }(\lambda v + GR) = (d+2)G,
\]
which implies that $G=0$.
\end{proof}

The foliation $\F$ can also be defined by a global section of
$\Omega_{\PP^2}^1\otimes \OO_{\PP^2}(d+2)$. By means of the dual of
the Euler sequence and Bott formulae we achieve the following
sequence$\colon$
\[
0 \longrightarrow \HH^0(\PP^2, \Omega_{\PP^2}^1\otimes
\OO_{\PP^2}(d+2) ) \longrightarrow \HH^0(\PP^2,
\OO_{\PP^2}(d+1))^{\oplus 3} \longrightarrow \HH^0(\PP^2,
\OO_{\PP^2}(d+2) ) \longrightarrow 0
\]
Then a global section of $\Omega_{\PP^2}^1\otimes \OO_{\PP^2}(d+2)$
is identified with a polynomial $1$-form
\[
\omega = A\dX + B\dY + C\dZ
\]
in $\C^3$ where $A$, $B$ and $C$ are homogeneous of degree $d+1$,
such that the contraction with $R$ is zero, that is, $XA+YB+ZC =0$.

\subsection{Finite Subgroups of $\aut(\PP^2)$}

We are interested in foliations with finite automorphism groups. Fortunately, the finite subgroups of $\aut(\PP^2) = \pgl(3,\C)$ have been classified by Blichfeldt, see \cite{BLI}.

\begin{theorem}[Blichfeldt]\label{bli}
Any finite subgroup of $\pgl(3,\C)$ is conjugated to one of the following.

Intransitive imprimitive groups:
\begin{itemize}
\item[(A)] An abelian group generated by diagonal matrices;
\item[(B)] A finite subgroup of ${\rm GL}(2,\C)$.
\end{itemize}

Transitive imprimitive groups:
\begin{itemize}
\item[(C)] A group generated by an abelian group and
\[
T(X:Y:Z) = (Y:Z:X);
\]
\item[(D)] A group generated by a group of type (C) and a transformation 
\[ 
R(X:Y:Z)=(aX:bZ:cY). 
\]
\end{itemize}

Primitive groups having a non-trivial normal subgroup:
\begin{itemize}
\item[(E)] The group of order $36$ generated by $S(X:Y:Z) = (X: \lambda Y: \lambda^2 Z )$, $T$ and 
\[
V(X:Y:Z) = (X+Y+Z:X+ \lambda Y+ \lambda^2 Z : X+ \lambda^2 Y+ \lambda Z),
\] 
where $\lambda$ is a primitive cubic root of $1$;
\item[(F)] The group of order $72$ generated by $S$, $T$, $V$ and $UVU^{-1}$ where 
\[
U(X:Y:Z)=(X:Y:\lambda Z);
\]
\item[(G)] The Hessian group of order $216$ generated by $S$, $T$, $V$, and $U$. 
\end{itemize}

Simple groups:
\begin{itemize}
\item[(H)] The icosahedral group isomorphic to $\mathfrak{A}_5$ of order $60$ generated by $T$ and the transformations $R_1(X:Y:Z)=(X:-Y:-Z)$ and 
\[
R_2 (X:Y:Z) = (X+\varphi Y -\varphi^{-1}Z: \varphi X -\varphi^{-1}Y +Z: -\varphi^{-1}X + Y + \varphi Z )
\]
where $\varphi = \frac{1}{2}(1+\sqrt{5})$.
\item[(I)] The Valentiner group isomorphic to $\mathfrak{A}_6$ of order $360$ generated by $T$, $R_1$, $R_2$ and $S$.
\item[(J)] The Klein group isomorphic to ${\rm PSL}(2, \mathbb{F}_7)$ of order $168$ generated by $T$ and the transformations $R^{\prime}(X:Y:Z) = (X:\beta Y: \beta^3 Z)$ and 
\[
S^{\prime}(X:Y:Z) = (aX+bY+cZ:bX+cY+aZ:cX+aY+bZ)
\]
where $\beta^7 = 1$, $a=\beta^4-\beta^3$, $b=\beta^2-\beta^5$ and $c=\beta - \beta^6$. 
\end{itemize}
\end{theorem}

Later Dolgachev and Iskovskikh refined the classification of the transitive
imprimitive groups, the groups of type (C) and (D), see \cite{DI}. If $G$ is a
group of one of these types, it has an abelian normal subgroup $H$
such that $G/H$ permutes transitively the three $H$--invariant
lines. It turns out that $G/H$ is isomorphic to either $\Z/3\Z$ or
the symmetric group $\mathcal{S}_3$ on three elements. They have
proved the following:

\begin{theorem}[Dolgachev--Iskovskikh]\label{trimp}
Let $G$ be a transitive imprimitive subgroup of $\pgl(3,\C)$. Then, $G$ is conjugate to  one of the following groups:
\begin{itemize}
\item[(C1)] $G \simeq (\Z/n\Z)^2 \rtimes (\Z/3\Z)$ generated by
\begin{align*}
S_1(X:Y:Z) & = (lX: Y: Z ),\\
S_2(X:Y:Z) & = (X: lY: Z ),\\
 T(X:Y:Z) & = (Y:Z:X) ,
\end{align*}
where $l$ is a primitive $n$--th root of the unity.
\item[(D1)] $G \simeq (\Z/n\Z)^2 \rtimes \mathcal{S}_3$ generated by $S_1$, $S_2$, $T$ and
\[
R(X:Y:Z) = (Y:X:Z).
\]
\item[(C2)] $G \simeq (\Z/n\Z \times \Z/m\Z) \rtimes (\Z/3\Z)$ generated by 
\begin{align*}
S_3(X:Y:Z) & = (l^kX: Y: Z ),\\
S_4(X:Y:Z) & = (l^sX: lY: Z ),\\
 T(X:Y:Z) & = (Y:Z:X) ,
\end{align*}
where $k>1$, $n=mk$ and $s^2-s+1 \equiv 0 \pmod{k}$.
\item[(D2)] $G \simeq (\Z/n\Z \times \Z/m\Z)\rtimes \mathcal{S}_3$ generated by $S_3$, $S_4$, $T$ and $R$, where $k=3$ and $s=2$.
\end{itemize}
\end{theorem}

Observe that groups of type $(A)$ or $(B)$ fix the point $(1:0:0)$ and all the others do not fix any point but have a subgroup of type $(C)$. In particular, they contain the element $T$. Then we have restrictions on foliations with automorphism groups without fixed points, as we establish in the following lemma:

\begin{lemma}\label{fixT}
Let $\F$ be a foliation on $\PP^2$ of degree $d\geq 0$ and let $G<
\aut(\F)$ be a finite subgroup. Then one of the following holds:
\begin{enumerate}
\item $G$ has a fixed point; 
\item Up to linear change of coordinates, $\aut(\F)$ contains
\[
T(X: Y: Z) = (Y: Z: X)
\]
and $\F$ is defined by a vector field of the form
\[
v = A(X,Y,Z)\partial_X + \lambda^2 A(Y,Z,X)\partial_Y + \lambda
A(Z,X,Y)\partial_Z,
\]
where $\lambda^3 = 1$ and $A$ is a homogeneous polynomial of degree
$d$.
\end{enumerate}
Moreover, if $\aut(\F)$ contains both $T$ and $R(X\colon Y\colon Z) = (Y\colon X\colon Z)$,
then $\lambda = 1$ and $A\circ T \circ R = \pm A$.
\end{lemma}

\begin{proof}
If $G$ does not fix any point in $\PP^2$, then up to conjugation we
can  suppose that $T \in G$. Let $v = A\partial_X + B\partial_Y +
C\partial_Z$ be a vector field defining $\F$ such that ${\rm div}(v)
= 0$. The vector field $T^{-1}_{\ast}v$ also defines $\F$ and it follows  from
Lemma \ref{div} that we have a relation $T^{-1}_{\ast}v = \lambda v$, 
for some $\lambda \in \C^{\ast}$. This means that
\[
(A\circ T)\partial_X + (B\circ T)\partial_Y + (C\circ T)\partial_Z =
\lambda(B\partial_X + C\partial_Y + A\partial_Z),
\]
which implies that $A\circ T = \lambda B$, $B\circ T = \lambda C$,
$C\circ T = \lambda A$, $\lambda^3 = 1$ and
\[
v = A(X,Y,Z)\partial_X + \lambda^2 A(Y,Z,X)\partial_Y + \lambda
A(Z,X,Y)\partial_Z.
\]

Suppose that $\aut(\F)$ has a subgroup $G$ of type (D). By Theorem
\ref{trimp}, we have that  $G$ is conjugate to a group of type (D1) or
(D2). Then, up to change of homogeneous coordinates, $G$ contains
$T$ and
\[
R(X: Y: Z) = (Y: X: Z).
\]

It follows that $R_{\ast}v = \alpha v$ and $(T\circ R)_{\ast}v =
\beta v$ with $\alpha^2=\beta^2 = 1$, since these automorphisms have
order two. However,
\[
\beta v = (T\circ R)_{\ast}v =T_{\ast}( R_{\ast}v) = \lambda^2\alpha
v
\]
hence $\lambda =1$ and $\alpha = \beta$. In particular,
\[
v = A(X,Y,Z)\partial_X + A(Y,Z,X)\partial_Y + A(Z,X,Y)\partial_Z.
\]
Calculating the pushforward by the automorphism $R$ we have
\[
R_{\ast}v = A(X,Z,Y)\partial_X + A(Y,X,Z)\partial_Y +
A(Z,Y,X)\partial_Z.
\]
Then $A(X,Z,Y) = \alpha A(X,Y,Z)$. That is, $A\circ T \circ R = \pm
A$.
\end{proof}

\subsection{Bounds}

Every nontrivial element $\varphi \in \pgl(3,\C)$, $\varphi \neq id$, of finite order is diagonalizable and either has three isolated fixed points or fixes
pointwise a line and a point outside of that line. An element in this last
case is called by  pseudo-reflection, and it  occurs when $\varphi$
has only one eigenvalue different from $1$. We begin our analysis by
these simple automorphisms.

\begin{proposition}\label{psref}
Let $\F$ be a foliation, of degree $d$, on $\PP^2$ invariant by an automorphism of the form 
\[
\varphi(X:Y:Z)=(X:Y:lZ),
\]
with  $l^n=1$ primitive, $n\geq 2$. Then $n\leq d+1$. Moreover, $(0:0:1) \not\in \sing(\F)$ if and only if one of the following holds:
\begin{enumerate}
\item $n \mid d$,  $\{Z=0\}$ is $\F$--invariant and the $\F$ is induced by a vector field of the form
$$
Z^{n-1}A(X,Y,Z^n)\partial_X + Z^{n-1}B(X,Y,Z^n)\partial_Y + C(X,Y,Z^n)\partial_Z;
$$
\item $n \mid d+1$,  $\{Z=0\}$ is not $\F$--invariant and the $\F$ is induced by a vector field of the form
$$
A(X,Y,Z^n)\partial_X + B(X,Y,Z^n)\partial_Y + ZC(X,Y,Z^n)\partial_Z.
$$
\end{enumerate}

\end{proposition}

\begin{proof}
Let $v = \widetilde{ A}\partial_X + \widetilde{ B}\partial_Y + \widetilde{C}\partial_Z$ be a vector field inducing $\F$ such that ${\rm div}(v) = 0$. Then, for some $\lambda \in \C^{\ast}$, we have 
\begin{equation}\label{pseq}
\left(\widetilde{ A}\circ \varphi\right)\partial_X + \left(\widetilde{ B}\circ \varphi\right)\partial_Y + \left(\widetilde{ C}\circ \varphi\right)\partial_Z = \lambda\left(\widetilde{ A}\partial_X + \widetilde{ B}\partial_Y + l\widetilde{ C}\partial_Z\right).
\end{equation}
This implies that $\widetilde{ A}\circ \varphi =\lambda \widetilde{ A}$, $\widetilde{ B}\circ \varphi =\lambda \widetilde{ B}$ and  $\widetilde{ C}\circ \varphi =\lambda l\widetilde{ C}$. We also have that $\lambda = l^r$ for some $r>0$, since $\varphi^n=id$ implies that $\lambda^n=1$. We can  expand these polynomials in terms of $Z\colon$
\begin{align*}
\widetilde{ A}(X,Y,Z) = \sum_{k}a_{d-k}(X,Y)Z^k,\\
\widetilde{ B}(X,Y,Z) = \sum_{k}b_{d-k}(X,Y)Z^k,\\
\widetilde{ C}(X,Y,Z) = \sum_{k}c_{d-k}(X,Y)Z^k,
\end{align*}
where $a_{d-k}$, $b_{d-k}$ and $c_{d-k}$ are homogeneous of degree $d-k$. Therefore, the equation (\ref{pseq}) implies that if $a_{d-k}\neq0$ or $b_{d-k}\neq0$ then $k\equiv r \pmod{n}$, and  if $c_{d-k}\neq0$ then $k\equiv r+1 \pmod{n}$. 

By definition, $\F$ has isolated singularities, hence $\widetilde{ A}$, $\widetilde{ B}$ and $\widetilde{C}$ are relative prime polynomials. In particular, $Z$ divides at most two of them. We separate our argument in two cases: whether $Z$ divides $\widetilde{C}$  or not.

First case: $Z$ does not divide $\widetilde{ C}$ ($\{Z=0\}$ is not $\F$--invariant).\\
We have that $c_d \neq 0$, hence $r\equiv -1 \pmod{n}$. This implies that $a_{d-k}=b_{d-k}=0$ for all $k$ such that $k\not\equiv -1 \pmod{n}$. Therefore, $Z^{n-1}$ divides $\widetilde{ A}$ and $\widetilde{ B}$. In particular, 
\[
n \leq d+1.
\]
In this case, we obtain 
\[
v = Z^{n-1} A(X,Y,Z^n)\partial_X + Z^{n-1}B(X,Y,Z^n)\partial_Y + C(X,Y,Z^n)\partial_Z.
\]

Second case: $Z$ divides $\widetilde{ C}$ ($\{Z=0\}$ is $\F$--invariant).\\
We have that $Z$ does not divide $\widetilde{ A}$ or $\widetilde{ B}$, hence $a_{d}\neq 0$ or $b_{d}\neq 0$. Either way, it follows that $r\equiv 0 \pmod{n}$. Therefore, $\widetilde{ A}$ and $\widetilde{ B}$ are $\varphi$--invariant, that is, they are polynomials in $X$, $Y$ and $Z^n$. In particular, 
\[
n\leq d.
\]
In this case, we obtain 
\[
v = A(X,Y,Z^n)\partial_X + B(X,Y,Z^n)\partial_Y + ZC(X,Y,Z^n)\partial_Z.
\]

Now, we analyze the foliation around the point $(0:0:1)$. In the chart $\{Z\neq 0\}$, let $(x,y) = (X/Z, Y/Z)$ be the affine  coordinates. Then $(0:0:1)$ corresponds to the origin and $\F$ is given by
\[
v(x,y) = [\widetilde{ A}(x,y,1) -x\widetilde{ C}(x,y,1)]\partial_x + [\widetilde{ B}(x,y,1) -y\widetilde{ C}(x,y,1)]\partial_y.
\]
Then $v$ has a singularity at the origin if and only if $\widetilde{ A}(0,0,1)=\widetilde{ B}(0,0,1)=0$. Therefore, $(0:0:1)\ \not\in \sing(\F)$ if and only if $a_0 \neq 0$ or $b_0 \neq 0$. In any case, we have $d\equiv r  \pmod{n}$. Then 
\begin{enumerate}
\item $d \equiv -1 \pmod{n}$ if $\{Z=0\}$ is not $\F$--invariant, or
\item $d \equiv 0 \pmod{n}$ if $\{Z=0\}$ is $\F$--invariant.
\end{enumerate}
\end{proof}

Our next step is to analyze the automorphisms that fix only regular
points. Proposition \ref{streg} gives a glimpse on what we might
find. Speaking of which, the global situation imposes even more restrictions.

\begin{theorem}\label{p2reg}
Let $\F$ be a foliation in $\PP^2$ of degree $d$. Let $G<\aut(\F)$ be a finite subgroup that has fixed points. If $G$ fixes only regular points, then it is cyclic and has a generator of the form
\[
\varphi(X:Y:Z) = (X:lY:l^{k+1}Z),
\]
where $l$ is a root of the unity, $k = \tg(\F,L_p, p)>0$ and $L_p$ is the line tangent to $\F$ at $p\notin \sing(\F)$  a fixed point. Moreover, 
\begin{itemize}
\item $|G|\mid (d^2 + d + 1)$ if all singularities have trivial stabilizers in $G$; or
\item $|G|\mid d(d+1)$ if there exists a singularity with nontrivial stabilizer.
\end{itemize}

\end{theorem}

\begin{proof}
Let  $p\notin \sing(\F)$ be a fixed point. Then, $G$ is abelian by Proposition \ref{streg}. We can choose coordinates such that $G$ is diagonal and fixes $p=(1:0:0)$, $(0:1:0)$ and $(0:0:1)$, and $\F$ is tangent to $\{Z=0\}$ at $p$. We can suppose also that $\{Z=0\}$ is not $\F$--invariant. Indeed, $G$ leaves three lines invariant: $\{X=0\}$, $\{Y=0\}$ and $\{Z=0\}$. If two of these lines were $\F$--invariant, there would exist a singularity fixed by $G$ at their intersection. Hence, at most one of these lines can be $\F$--invariant. Up to permuting coordinates we can choose $\{Z=0\}$ to be not $\F$--invariant. In this case, Proposition \ref{tgreg} implies that $G$ is cyclic and is generated by 
\[
\varphi(X:Y:Z) = (X:lY:l^{k+1}Z),
\]
where $l$ is a root of the unity and $k = \tg(\F,\{Z=0\}, p)>0$. 

We can decompose $\sing(\F)$ into disjoint $G$--orbits. It follows from Proposition \ref{orb} that 
\[
d^2 + d + 1 = c_2(T\PP^2\otimes\OO_{\PP^2}(d-1)) =  \sum_i \frac{|G|}{|H_i|}\mu(\F,t_i), 
\]
where the $t_i \in \sing(\F)$ lie on disjoint orbits and $H_i$ is the stabilizer of $t_i$. If all singularities have trivial stabilizers, then 
\[
|G| \mid (d^2 + d + 1).
\]
Now, suppose  that there exists a singularity with nontrivial stabilizer, say $t_1$. Every element of $H_1$ fixes more than three points, hence $H_1$ is cyclic and generated by a pseudo-reflection $\psi = \varphi^a$,  with $a = \left(G:H_1\right)>0$. We have two distinct cases:
\begin{enumerate}
\item $|H_1|=\gcd(k+1, |G|)>1$ and $H_1$ is generated by
\[
\psi(X:Y:Z) = (X:l^aY:Z).
\]
\item $|H_1|=\gcd(k, |G|)>1$ and $H_1$ is generated by
\[
\psi(X:Y:Z) = (l^{-a}X:Y:Z).
\] 
\end{enumerate}

For the first case, the line $\{Y=0\}$ is not $\F$--invariant since $\F$ is tangent to $\{Z=0\}$ at $p$. By the Proposition \ref{psref}, $|H_1|$ divides $d+1$  and $\F$ is given by a vector field of the form
\[
v = Y^{u-1}A(X,Y^u,Z)\partial_X + B(X,Y^u,Z)\partial_Y + Y^{u-1}C(X,Y^u,Z)\partial_Z,
\] 
where $u=|H_1|$. It is straightforward to verify that, in this case, $\F$ is tangent to $\{X=0\}$ at $(0:0:1)$. The line $\{Y=0\}$ is $G$--invariant, since $G$ fixes $(0:0:1)$ and $(1:0:0)$, and we have seen that $\F$ is not tangent to $\{Y=0\}$ at these two points. Since $H_1$ is the stabilizer for every other point in $\{Y=0\}$, Proposition \ref{orbind} implies
\[
d = \tg(\F, \{Y=0\}) =  \left(G:H_1\right) \sum_{t\in \{Y=0\}}  \tg(\F, \{Y=0\}, t).
\]
Then $\left(G:H_1\right)$ divides $d$ and we have
\[
|G|  = \left(G:H_1\right)|H_1| \mid d(d+1).
\]

We may separate the second case into two subcases: whether $\{X=0\}$ is $\F$--invariant or not. First suppose that $\{X=0\}$ is not $\F$--invariant. By the Proposition \ref{psref}, $|H_1|$ divides $d+1$ and $\F$ is given by a vector field of the form
\[
v = A(X^u,Y,Z)\partial_X + X^{u-1}B(X^u,Y,Z)\partial_Y + X^{u-1}C(X^u,Y,Z)\partial_Z,
\] 
where $u=|H_1|$. Then $\F$ is tangent to $\{Y=0\}$ at $(0:0:1)$ and to $\{Z=0\}$ at $(0:1:0)$. Therefore, the proof follows the argument of the first case.

Now suppose that $\{X=0\}$ is $\F$--invariant. Proposition \ref{psref} implies that $|H_1|$ divides $d$ and $\F$ is given by a vector field of the form
\[
v = XA(X^u,Y,Z)\partial_X + B(X^u,Y,Z)\partial_Y + C(X^u,Y,Z)\partial_Z,
\]   
where $u=|H_1|$. Since the points $(0:0:1)$ and $(0:1:0)$ are regular and $H_1$ is the stabilizer of any other point in $\{X=0\}$, Proposition \ref{orbind} implies
\[
d+1 = {\rm Z}(\F, \{X=0\}) =\left(G:H_1\right) \sum_{t\in \{X=0\}} {\rm Z}(\F, \{X=0\}, t).
\]
Then $\left(G:H_1\right)$ divides $d+1$ and we have
\[
|G|  = \left(G:H_1\right)|H_1| \mid d(d+1).
\]
\end{proof}

Now we turn our attention to the groups that fix a singularity.
We restrict ourselves to the reduced ones.

\begin{proposition}\label{p2A}
Let $\F$ be a foliation in $\PP^2$ of degree $d\geq 3$. Let
$G<\aut(\F)$ be a finite subgroup that fixes a reduced singularity
$p$ which does not have a line as a separatrix. Then
\[
|G| \leq 2(d^2-1) .
\]
Moreover,  
\[
|G| \leq (d^2-1)
\]
if $G$ is abelian.
\end{proposition}

\begin{proof}
It follows from  Theorem \ref{bli} that $G$ is of type $(A)$ or $(B)$, since it has a fixed point. Suppose that $G$ is of type $(A)$. Up to a change of coordinates $G$ is diagonal and $p=(0 : 0 : 1)$. Then $G$ has three invariant lines: $\{X=0\}$, $\{Y=0\}$ and $\{Z=0\}$. By hypothesis, $\{X=0\}$ and $\{Y=0\}$ are not $\F$--invariant. Hence the Proposition \ref{fibtg} implies that $G$ falls in one the following cases:
\begin{enumerate}
\item $G$ is cyclic;
\item $G$ has a cyclic subgroup generated by a pseudo-reflection whose index is two.
\end{enumerate}
In the second case, Proposition \ref{psref} implies that 
\[
|G| \leq 2(d+1) \leq d^2-1,
\]
since $d\geq 3$. In the first case, Proposition \ref{fibtg} also shows that $G$ is generated by 
\[
\varphi(X: Y: Z) = (l^kX: lY: Z),
\]
where $l^{|G|} = 1$ and $k = \tg(\F, \{X=0\}, p)$ if we consider the line $\{X=0\}$, and
\[
\psi(X: Y: Z) = (\zeta X: \zeta^q Y: Z),
\]
where $\zeta^{|G|} = 1$ and $q = \tg(\F, \{Y=0\}, p)$ if we consider the line $\{Y=0\}$. Therefore,  we have that $\psi = \varphi^r$ for some $r$ such that  $\gcd(r,|G|)=1 $. This implies that $\zeta= l^r$ and $kq \equiv 1 \pmod{|G|}$. If $k \equiv 1 \pmod{|G|}$, then $\varphi$ is a pseudo-reflection. It follows form  Proposition \ref{psref} that 
\[
|G| \leq d+1 < d^2-1.
\] 
If $k \not\equiv 1 \pmod{|G|}$ then
\[
|G| \leq kq-1 \leq d^2-1.
\]

If $G$ is not abelian (type $(B)$) then the Proposition \ref{stab} implies that $G$ has an abelian subgroup of index two. Therefore,
\[
|G|\leq 2(d^2-1).
\]
\end{proof}

The groups that do not fix any point contain the automorphism $T$
described in Theorem \ref{bli}. For the simplest case, type (C1), we
have the following:

\begin{proposition}\label{psrefT}
Let $\F$ be a foliation in $\PP^2$ of degree $d$. Suppose that, for some
choice of coordinates, $\aut(\F)$ contains $T$ and a diagonal
pseudo-reflection $\varphi$ of order $n$, i.e, $\aut(\F)$
contains a subgroup of type (C1). Then,  either:
\begin{itemize}
\item $n \mid (d-1)$ if the curve $\{XYZ=0\}$ is $\F$--invariant, or
\item $n \mid (d+2)$ if the curve $\{XYZ=0\}$ is not $\F$--invariant.
\end{itemize}
Moreover, $\F$ always has singularities at the points $(0:0: 1)$,
$(0:1: 0)$ and $(1: 0: 0)$. In the second case, these singularities are
reduced only if $n=2$.
\end{proposition}

\begin{proof}
By Lemma \ref{fixT}, we have that $\F$ is defined by a homogeneous vector field
\[
v =  A(X,Y,Z)\partial_X + \lambda^2 A(Y,Z,X)\partial_Y + \lambda A(Z,X,Y)\partial_Z
\]
where $\lambda^3 = 1$, since $T\in \aut(\F)$. If we write $\varphi(X:Y:Z) = (X:Y:\zeta Z)$, $\zeta^n=1$, then 
\begin{align*}
T\circ \varphi\circ T^2(X:Y:Z) & = (X:\zeta Y: Z),\\
T^2\circ \varphi\circ T(X:Y:Z) & = (\zeta X: Y: Z) 
\end{align*}
also belong to $\aut(\F)$. 

One of  the lines $\{X=0\}$, $\{Y=0\}$ or $\{Z=0\}$ is $\F$--invariant if and only if the other two are $\F$--invariant, since they are permuted by $T$. We fall in two cases: whether these three lines are $\F$--invariant or not.

First case: $\{XYZ=0\}$ is $\F$--invariant.\\
By Proposition \ref{psref}, we can write
\[
v = X\widetilde{A}(X^n,Y^n,Z^n)\partial_X + \lambda^2 Y\widetilde{A}(Y^n,Z^n,X^n)\partial_Y + \lambda Z\widetilde{A}(Z^n,X^n,Y^n)\partial_Z.
\]
In particular, $\widetilde{A}(X,Y,Z)$ is a homogeneous polynomial of degree $k$, that satisfies $d=nk+1$. Hence,
\[
n \mid (d-1).
\]
 
Second case: $\{XYZ=0\}$ is not $\F$--invariant.\\
By  Proposition \ref{psref}, we can write
\begin{align*}
v  = \, &  Y^{n-1}Z^{n-1}\widetilde{A}(X^n,Y^n,Z^n)\partial_X + \lambda^2 X^{n-1}Z^{n-1}\widetilde{A}(Y^n,Z^n,X^n)\partial_Y \\ 
& + \lambda X^{n-1}Y^{n-1}\widetilde{A}(Z^n,X^n,Y^n)\partial_Z.
\end{align*}
In particular, $\widetilde{A}(X,Y,Z)$ is a homogeneous polynomial of degree $k$ that satisfies $d+2=n(k+2)$. Hence,
\[
n \mid (d+2).
\]
On the affine chart $\{Z\neq 0\}$ with the standard coordinates $(x,y) = (X/Z,Y/Z)$, $v$ is written as
\begin{align*}
v = \, & y^{n-1}\left(\widetilde{A}(x^n,y^n,1)- x^n\lambda^2\widetilde{A}(1,x^n,y^n)\right)\partial_x \\
& + \lambda x^{n-1}\left(\lambda\widetilde{A}(y^n,1,x^n)- y^n\widetilde{A}(1,x^n,y^n)\right)\partial_y.
\end{align*}
If the singularity at $(0:0:1)$ is reduced then $n=2$ and 
\[
\widetilde{A}(0,1,0)\widetilde{A}(0,0,1)\neq 0. 
\]
\end{proof}

To conclude this section we state a theorem combining the results
that we have proved so far.

\begin{theorem}\label{bdp2}
Let $\F$ be a foliation in $\PP^2$, of degree $d\geq 3$, such that $\aut(\F)$ is finite and imprimitive. That is,  $\aut(\F)$ leaves invariant the union of three lines, $L_1$, $L_2$ and $L_3$,  in general position (meeting in three distinct points). If these lines are not $\F$--invariant and support at most reduced singularities, then 
\[
|\aut(\F)| \leq 3(d^2+d+1)
\]
\end{theorem}

\begin{proof}
Up to a linear change of coordinates, we may suppose that $L_1$, $L_2$ and $L_3$ are $\{X=0\}$, $\{Y=0\}$ and $\{Z=0\}$ and $\aut(\F)$ is in one of the forms in Theorem \ref{bli} from (A) to (D). 

First case: $\aut(\F)$ is of type (A).\\ 
The points $(0:0:1)$, $(0:1:0)$ and $(1:0:0)$ are regular or at least one of them is a singularity. If they are regular, by Theorem \ref{p2reg}, 
\begin{equation}\label{tA1}
|\aut(\F)| \leq d^2+d+1.
\end{equation}
If at least one of these points belongs to $\sing(\F)$, Proposition \ref{p2A} implies that
\begin{equation}\label{tA2}
|\aut(\F)| \leq d^2-1,
\end{equation}
since the $G$--invariant lines that contain this point are not $\F$--invariant. Comparing the inequalities (\ref{tA1}) and (\ref{tA2}) we have the bound for type (A):
\begin{equation}\label{tA}
|\aut(\F)| \leq d^2+d+1.
\end{equation}

Second case: $\aut(\F)$ is of type (B).\\
The group $\aut(\F)$ fixes a singularity at $(1:0:0)$ (that is reduced, by hypothesis) and has a subgroup of type (A) whose index is two. Hence we double the previous bound:
\begin{equation}\label{tB}
|\aut(\F)| \leq 2(d^2-1).
\end{equation}

Third case: $\aut(\F)$ is of type (C).\\ 
By Theorem \ref{trimp}, we fall in two cases: (C1) or (C2). For groups of type (C1),  
\[
\aut(\F) = (\Z/n\Z)^2 \rtimes (\Z/3\Z).
\]
generated by 
\[
T(X:Y:Z) = (Y:Z:X)
\]
and diagonal pseudo-reflections. Since $\{XYZ=0\}$ is not $\F$--invariant and supports only reduced singularities, Proposition \ref{psrefT} implies that $n=2$. Hence,
\begin{equation}\label{tC1}
|\aut(\F)| = 12.
\end{equation}

For groups of type (C2), we have that 
\[
\aut(\F)\simeq (\Z/n\Z\times\Z/m\Z) \rtimes (\Z/3\Z)
\]
generated by $T$,
\begin{align*}
\varphi(X:Y:Z) & = (l^kX: Y: Z ) \, {\rm and}\\
\psi(X:Y:Z) & = (l^sX: lY: Z ), 
\end{align*}
where $l^n=1$, $k>1$, $n=mk$ and $s^2-s+1 \equiv 0 \pmod{k}$. Observe that $\varphi$ is a pseudo-reflection of order $m$. We have two subcases: $k<n$ and $k=n$.

Under our hypotheses, if $k<n$ then Proposition \ref{psrefT} implies that $m=2$. Hence $n=2k$ and $\F$ is given by a vector field of the form
\begin{align*}
v  = \, &  YZA(X^2,Y^2,Z^2)\partial_X + \lambda^2 XZA(Y^2,Z^2,X^2)\partial_Y \\ 
& + \lambda XYA(Z^2,X^2,Y^2)\partial_Z.
\end{align*}
We have that $d\geq 3$ and $A$ satisfies
\[
d = 2\deg(A) + 2,
\]
which implies that $\deg(A) \geq 1$. Since $\{X=0\}$ is not $\F$--invariant, $X$ cannot divide $A$. Then $A(X^2,Y^2,Z^2)$ has monomials $gY^{d-2}$ and $hZ^{d-2}$ for some $g,h\in \C$, with $g\neq 0$ or $h\neq0$. The vector field
\[
w = \left(gY^{d-1}Z + hYZ^{d-1}\right)\partial_X + \lambda^2\left(gZ^{d-1}X + hZX^{d-1}\right)\partial_Y +\lambda \left(gX^{d-1}Y + hXY^{d-1}\right)\partial_Z
\]
is also invariant by $\aut(\F)$. Suppose that $g\neq 0$. Taking the pushforward of $w$ by $\psi$ we see that 
\[
s+1-d \equiv s-1 \equiv s-sd-1 \pmod{n}.
\]
Hence, $n$ divides $d-2$. Therefore,
\begin{equation}\label{tC21}
|\aut(\F)| = 3nm= 6n \leq 6(d-2).
\end{equation}
If $g=0$, then $h\neq 0$. Taking the pushforward of $w$ by $\psi$ we see that 
\[
s-1 \equiv 1-s(d-1) \equiv -s-d+1 \pmod{n}.
\]
Then we have that $d\equiv 2(1-s)  \pmod{n}$. Since $d = 2\deg(A) + 2$ and $n=2k$,
\[
\deg(A) \equiv -s \pmod{k},
\]
which implies that
\[
\deg(A)^2+\deg(A)+1 \equiv s^2-s+1 \equiv 0 \pmod{k}.
\]
Therefore,
\begin{equation}\label{tC21a}
|\aut(\F)| = 3nm= 12k \leq 12\left(\deg(A)^2+\deg(A)+1\right) = 3d^2-6d+12 
\end{equation}

Now, suppose that $k=n$. We have that 
\[
\aut(\F)\simeq (\Z/n\Z) \rtimes (\Z/3\Z),
\]
which is generated by $T$ and
\[
\psi(X:Y:Z) = (l^sX: lY: Z ),
\]
where $l^n=1$ and $s^2-s+1 \equiv 0 \pmod{n}$. Hence, $\aut(\F)$ has a (normal) cyclic subgroup of index 3. By the inequality (\ref{tA}), we have that
\begin{equation}\label{tC22}
|\aut(\F)| \leq 3(d^2+d+1).
\end{equation}
Comparing (\ref{tC1}), (\ref{tC21}), (\ref{tC21a}) and (\ref{tC22}), we have:
\begin{equation}\label{tC}
|\aut(\F)| \leq 3(d^2+d+1).
\end{equation}

Fourth case: $\aut(\F)$ is of type (D).\\ 
There exist two cases described by Theorem \ref{trimp}: (D1) and (D2). If $\aut(\F)$ is of type (D1), it has a subgroup of type (C1) whose index is two. It follows from (\ref{tC1}) that
\begin{equation}\label{tD1}
|\aut(\F)| = 24 < 3(d^2+d+1),
\end{equation}
since $d\geq3$.

Now suppose that $\aut(\F)$ is of type (D2). Then it has a subgroup $G$ of type (C2) where $k=3$ and $s=2$. We have seen in the third case that $n=k=3$ or $n=2k=6$. If $n=k=3$, then
\begin{equation}\label{tD21}
|\aut(\F)| = 2|G| =\frac{6n^2}{k}= 18 < 3(d^2+d+1),
\end{equation}
since $d\geq3$. If $n=2k=6$, then $G$ has order $36$. Applying (\ref{tC21a}) we obtain $d\geq 5$. Therefore,
\begin{equation} \label{tD22}
|\aut(\F)| = 2|G| = \frac{6n^2}{k}=  72 < 3(d^2+d+1).
\end{equation}

Comparing the inequalities (\ref{tA}), (\ref{tB}), (\ref{tC}), (\ref{tD1}), (\ref{tD21}) and (\ref{tD22}), we conclude that
\[
|\aut(\F)| \leq 3(d^2+d+1).
\]
\end{proof}

We will see that this bound is sharp. It turns out that it is
achieved by the family of examples given by Jouanoulou in \cite{J}.

\begin{example}
Let $\mathcal{J}_d$ be the Jouanolou's foliation in $\PP^2$, of degree $d\geq 3$, which is given by 
\[
v = Z^d\partial_X  + X^d\partial_Y + Y^d\partial_Z.
\]
The automorphism group of $\mathcal{J}_d$ is 
\[
\aut(\mathcal{J}_d) \simeq \Z/(d^2 +d + 1)\Z\rtimes \Z/3\Z,
\]
generated by 
\begin{align*}
T(X:Y:Z) & = (Y:X:Z),\\
\varphi(X:Y:Z) & = (X:lY:l^{d+1}Z), 
\end{align*}
where $l^{d^2+d+1}=1$, with $l$ primitive. The group $\aut(\mathcal{J}_d)$ is of type (C2), with $n=k=d^2+d+1$ and $s=-d$. 
\end{example}
Observe that, under the conditions of Theorem \ref{bdp2}, we obtain 
a linear bound on $K_{\F}^2$. Indeed,
in this case we get 
\[
|\aut(\F)| \leq 3 (d^2+d+1)\leq 21 K_{\F}^2,
\]
since $K_{\F}=\mathcal{O}_{\PP^2}(d-1)$.

\section{Bounds for Foliations on Geometrically Ruled Surfaces.}

A geometrically ruled surface is a surface $X$ together with a
surjective holomorphic map $\pi \colon X \longrightarrow C$ onto a smooth
compact curve $C$ such that the fiber $\pi^{-1}(y)$ is isomorphic to
$\PP^1$, for every $y\in C$.

For every geometrically ruled surface $\pi \colon X \longrightarrow C$,
there exists a rank two vector bundle $E$ over $C$ such that $X\simeq
\PP(E)$. $E$ is not uniquely determined but, if $E^{\prime}$ is
another vector bundle over $C$ such that $X\simeq \PP(E^{\prime})$,
then $E\simeq E^{\prime} \otimes L$ for some line bundle $L$ over
$C$.

Let $C_0\subset X$ be (the image of) a section. Then
\[
{\rm Pic}(X) \simeq \mathbb{Z}\cdot C_0 \oplus \pi^{\ast}{\rm
Pic}(C)
\]
and
\[
{\rm Num}(X) \simeq \mathbb{Z}\cdot C_0 \oplus \mathbb{Z}\cdot f, 
\]
where $f$ is the class of a fiber. The self-intersection numbers of the sections are bounded below and its minimum defines an invariant $e$ of $X$, defined by 
\[
e \ceq - \min\{D^2\mid D \ {\rm is \, (the \, image \, of)\, a \,
section \, of \,} \pi\}.
\]
A section $D$ is called minimal if $D^2 = -e$. We will assume that
the generator $C_0$ of ${\rm Num}(X)$ is the class of a minimal
section.

Equivalently, we can define $e$ by the line sub-bundles of $E$. Indeed, there exists a bijection between line subbundles of $E$ and sections of $\pi$, see \cite[Lemma 1.14]{MARC}. The degree of these
subbundles are bounded above and the maximum is exactly $e$.

\begin{definition}
If $L$ is a divisor on $X$, the class of $L$ will also be denoted $L$
and will be expressed by $L \equiv aC_0 + bf$, where $\equiv$ stands for
numerical equivalence. We will call $(a,b)$ the bidegree of $L$. A
foliation $\F$ on $X$ has bidegree $(a,b)$ if $K_{\F} \equiv aC_0 +
bf$.
\end{definition}

By applying the adjunction formula it follows that the canonical class
of $X$ is $K_X \equiv -2C_0 + (2g-2-e)f$. The tangent bundle of $X$
has a line subbundle $\tau$ defined by the kernel of the jacobian
of $\pi$,
\[
0 \longrightarrow \tau \longrightarrow T_X \longrightarrow
\pi^{\ast}T_C = N \longrightarrow 0, 
\]
where $N$ is the normal bundle of the ruling. Therefore, we have $K_X \simeq
\tau^{\ast}\otimes N^{\ast}$. From $\pi^{\ast}T_C = N$, we know that
$N \equiv (2-2g)f$, hence $\tau \equiv 2C_0 + ef$.

\begin{proposition}\cite[V, Propositions 2.20, 2.21]{HART}\label{nefp1}
Let $X$ be a ruled surface over a curve $C$ with invariant $e$. Let $Y \equiv aC_0 + bf$ be an irreducible curve different from $C_0$ or a fiber. Then:
\begin{itemize}
\item[a)] if $e\geq 0$, then $a>0$ and $b\geq ae$;
\item[b)] if $e<0$, then either $a=1$, $b\geq 0$ or $a\geq 2$, $2b\geq ae$.
\end{itemize}
A divisor $D \equiv aC_0 + bf$ is nef if and only if
\begin{itemize}
\item[a)] $e\geq 0$, $a\geq0$ and $b\geq ae$; 
\item[b)] $e<0$, $a\geq 0$ and $2b\geq ae$.
\end{itemize}
$D$ is ample when the inequalities are strict.
\end{proposition}
If $e> 0$ and $D$ is an effective divisor whose support does not contain $C_0$, then $D$ is nef. For a general effective divisor, $D\equiv nC_0 + E$ for some nef divisor $E$ and $n\geq 0$. If $e\leq0$ then every effective divisor is nef.

Now we will follow \cite{GOM} to recall some properties of foliations on geometrically ruled surfaces. However, our notation is different and we have refined some results. Let $\F$ be a foliation of bidegree $(a,b)$ on $X\longrightarrow C$
with invariant $e$ and $g(C) = g$. Then,
\begin{enumerate}
\item $\#\sing(\F) = c_2(T_X\otimes K_{\F}) = (a+1)(2b-ae + 2-2g)+ 2-2g$; \\
\item $N_{\F}^2 = (a+2)(2b-ae+4-4g)$; \\
\item $\tg(\F,F) = K_{\F}\cdot f + f^2 =a$, for not $\F$-invariant fiber $F$; \\
\item ${\rm Z}(\F,F)= \chi(F) + K_{\F}\cdot f = 2+a$ and ${\rm CS}(\F,F) = f^2 = 0$, for an $\F$-invariant fiber $F$.
\end{enumerate}
Since $c_2(T_X\otimes K_{\F})$ is the number of singularities of $\F$, it is a nonnegative integer. Hence 
\[
a \geq 0 \, {\rm and} \, 2b-ae \geq 2g-2,
\]
if $g\geq 1$. The last equality only holds if $g=1$ and $2b=ae$.

Suppose that $\F$ is not tangent to the fibers of the ruling. Let $\{U_i\}$ be a fine open cover of $X$ such that exist vector
fields $v_i$ defining $\F$ and $1$-forms $\omega_i$ defining the
ruling (regarded as a foliation). The holomorphic functions $g_i =
\omega_i(v_i)$ define a nontrivial divisor on $X$. This is the
tangency divisor between $\F$ and the ruling, denoted by $\tg(\F,
\tau)$. By construction, we have that
\[
\tg(\F, \tau) \sim K_{F} + N \equiv aC_0 + (b+2-2g)f
\]
and it is effective. In order to analyze $\tg(\F,\tau)$ we need a
technical lemma.

\begin{lemma}
Let $\F$ be a foliation on a smooth compact complex surface and
$C_1, \dots, C_k$ be disjoint smooth $\F$-invariant curves. If $D \sim
C_1+\dots + C_k$, then
\[
c_2(T_X(-\log D)\otimes K_{\F}) = \#\sing(\F) - \sum_{i=1}^{k}{\rm
Z}(\F, C_i).
\]
\end{lemma}

\begin{proof}
The curves $C_i$ being smooth and disjoint implies that $T_X(-\log
D)$ is locally free, by Saito's criterion, and fits into an exact
sequence:
\[
o \longrightarrow T_X(-\log D) \longrightarrow T_X \longrightarrow
\OO_D(D) \longrightarrow 0,
\]
see \cite{Liao} for details. Taking the total Chern class we have
that
\begin{align*}
c_1(T_X(-\log D)) & = -K_X-D,\\
c_2(T_X(-\log D)) & = c_2(X)+K_X\cdot D + D^2.
\end{align*}
Hence, by direct calculation we can show that
\begin{align*}
c_2(T_X(-\log D)\otimes K_{\F}) & = c_2(T_X(-\log D)) + c_1(T_X(-\log D))\cdot K_{\F} + K_{\F}^2 \\
& = c_2(X) +K_X\cdot D + D^2 + -(K_X+D)\cdot K_{\F} + K_{\F}^2 \\
& = c_2(T_X\otimes K_{\F}) -\chi(D) -K_{\F}\cdot D.
\end{align*}
The lemma follows from the index formulae (\ref{BB}) and (\ref{gsv}).
\end{proof}

In  \cite{CM} the authors proved   more general results for 
an one-dimensional foliation $\F$ on a compact complex manifold $X$, of
dimension $n$, with isolated singularities and an $\F$--invariant
 hypersurface $S$. They proved  under mild hypotheses on the singularities that lie
on $S$ that 
\[
\int_X c_n(T_X(-\log S)\otimes  K_{\F}) = \sum_{p \in
\sing(\F)\cap\{X\backslash S\}}\mu(\F,p)
\]
The restriction that is imposed is the vanishing of a logarithmic
index ${\rm Ind}_{log S, p}$ that in our case is
\[
{\rm Ind}_{log S, p} = \mu(\F,p) - {\rm Z}(\F,S,p).
\]
For reduced singularities, direct calculation shows that the
vanishing of this index holds for any separatrix of a nondegenerate 
singularity and for the weak separatrix of a saddle-node (when it
converges). However, it fails for the strong separatrix. We state a particular case of this result that will serve our
purposes:

\begin{lemma}\label{logc2}
Let $X$ be a compact complex surface and let $\F$ be foliation on
$X$. If $S$ is a smooth $\F$--invariant curve such that $\sing(\F)
\subset S$ and $\mu(\F,p) = {\rm Z}(\F,S,p)$ for all $p\in
\sing(\F)$, then
\[
c_2(T_X(-\log S)\otimes K_{\F})=0.
\]
\end{lemma}

This result applied to the analysis of the invariant fibers for a
foliation leads to the following classification:

\begin{theorem}\label{p1tg}
Let $X\longrightarrow C$ be a $\PP^1$-bundle over a smooth curve
whose genus is $g$ and let $e$ denote the invariant of $X$. Let $\F$
be a foliation on $X$ of bidegree $(a,b)$ such that $\mu(\F,p) =
{\rm Z}(\F,F,p)$ for every $p$ in a $\F$--invariant fiber $F$. Then
one of the following is true:
\begin{enumerate}
\item $\F$ is tangent to the ruling;
\item $a=0$ and $\F$ is Riccati;
\item $b=e=0$, $a>0$, $g=1$ and $\F$ is, up to a unramified cover, a regular foliation on $C\times \PP^1$;
\item $\F$ is a foliation on $\PP^1\times \PP^1$ with $a>0$, $b=0$ and all singularities lie on the two $\F$-invariant fibers. In particular, $\F$ is Riccati with respect to the other projection;
\item $e>0$, $b=(a+1)e$ and all singularities lie on $\F$-invariant fibers;
\item $\F$ has a singularity that lies on a fiber not $\F$-invariant.
\end{enumerate}
\end{theorem}

\begin{proof}
First suppose that $\F$ is not tangent to the ruling, then a
general fiber $F$ is not $\F$--invariant and $a= \tg(\F,F) \geq 0$.
Equality holds if and only if $\F$ is a Riccati foliation. This proves the cases  
$(1)$ and $(2)$.

Now, suppose that $a>0$. The generic transversality between $\F$ and a
general fiber implies that $\tg(\F, \tau)$ is a nontrivial effective
divisor. A fiber $F$ in the support of $\tg(\F, \tau)$ is
$\F$--invariant. Let $F_1, \cdots, F_s$ be the $\F$--invariant fibers and define
\begin{align*}
D & \ceq F_1 + \cdots + F_s \equiv sf, \\
\Delta & \ceq \tg(\F, \tau) -D \equiv aC_0 + (b + 2 -2g - s)f.
\end{align*}
By construction, the divisors $D$ and $\Delta$ are effective. 
Since $\F$  is neither Riccati or tangent to the ruling, the divisor $\Delta$ is not trivial.

Suppose that every singularity of $\F$ lies on an invariant fiber.
Then, by Lemma \ref{logc2}, we have 
\begin{align*}
0 & = c_2(T_X(-{\rm log}(D))\otimes K_{\F}) \\
& = (a+1)(2b-ae +2-2g) +2-2g -s(a+2)\\
& = (a+2)(b+2-2g-s) + a(b-ae-e).
\end{align*}
Hence, $(a+2)(b+2-2g-s) = - a(b-ae-e)$ and
\[
\Delta = K_{\F}+N-D \equiv aC_0 + \frac{- a(b-ae-e)}{a+2}f.
\]

If $e\leq 0$, every effective divisor is nef. It follows that
$-2a(b-ae-e) \geq (a+2)ae$, and this implies that
\[
0\geq e \geq 2b-ae \geq 2g-2 \geq 0,
\]
when $g\geq 1$. Hence, $2b=e=0$, $g=1$ and $\F$ is a regular foliation on surface ruled over an elliptic curve $C$. By \cite[Theorem 3.3]{GOM}, there exists an unramified cover $C\times \PP^1 \longrightarrow X$ and $\F$ lifts to a foliation described in \cite[Proposition 3.2]{GOM}. This proves the case $(3)$.

If $g=0$,  then $e\geq 0$ always. Indeed, a theorem due to Nagata says
that $e \geq -g$, see \cite[p.384]{HART} or \cite{MARC}. Hence,
$e=0$ and $0\geq 2b\geq -2$, which implies that $b=0,-1$ and
$X=\PP^1\times \PP^1$. On the other hand, $(a+2)$ divides $ab$,
since $\Delta$ is an effective integral divisor, then $b=0$ and,
consequently, $s=2$. In particular, $\F$ is Riccati with respect to
the other projection. This gives us the case $(4)$.

If $e>0$, then $C_0$ has to be contained in the support of $\Delta$. Indeed, if $C_0$ were not in the support of $\Delta$, then $\Delta$ would be nef, which implies that $- a(b-ae-e) \geq (a+2)ae$. Then
\[
0> -2e \geq 2b \geq ae >0,
\]
which is an absurd. The section $C_0$ being in the support of $\Delta$
means that $\F$ is regular at a general point $p\in C_0$ and the leaf that passes through $p$ is transverse to $C_0$. Thus
\[
0 \leq \tg(\F, C_0) = K_{\F}\cdot C_0 + C_0^2 = b-ae-e.
\]
On the other hand, $\Delta = mC_0 + E$, $E$ a nef divisor. Then
\[
E \equiv (a-m)C_0 + \frac{ -a(b-ae-e)}{a+2}f
\]
and we have that $(a-m)\geq 0$ and $-a(b-ae-e) \geq (a-m)(a+2)e$.
Hence
\[
0 \geq -a(b-ae-e) \geq (a-m)(a+2)e \geq 0.
\]
Consequently, $a=m$, $b=(a+1)e$, $\Delta = aC_0$ and there exist $s=b+2-2g$ invariant fibers. This proves case $(5)$.

If $\F$ does not fit in any of the previous cases, then it has a singularity on a generically transverse fiber.
\end{proof}

\subsection{Finite Automorphism Groups of Ruled Surfaces}

The automorphism groups of geometrically ruled surfaces have been
classified by Maruyama in \cite{MARA}. As one may expect, it is
closely related to automorphisms of rank two vector bundles over a
curve. Indeed, for a ruled surface $\pi\colon X \longrightarrow C$,
Maruyama proved that if either $C$ is irrational or $C$ is rational
and $X\not\simeq \PP^1 \times \PP^1$, then $\aut(X)$ fits in a exact
sequence
\[
1 \longrightarrow \aut_C(X) \longrightarrow \aut(X) \longrightarrow
\aut(C),
\]
where $\aut_C(X)$ is the subgroup that sends each fiber of $\pi$
onto itself. Alternatively, it is the automorphism group of $X$ seen
as a scheme over $C$. A lemma due to Grothendieck relates this
subgroup to the automorphism group of a vector bundle $E$ such that
$X\simeq \PP(E)$ by the following exact sequence:
\[
1 \longrightarrow \frac{\aut(E)}{\HH^0(C, \mathcal{O}_C^{\ast})}
\longrightarrow \aut_C(X) \longrightarrow \Delta \longrightarrow 1,
\]
where $\Delta = \{N \in {\rm Pic}(C) \mid E \simeq E\otimes N \}$
and, in our case, $\HH^0(C, \mathcal{O}_C^{\ast}) = \C^{\ast}$ acts
on $\aut(E)$ by rescaling. See \cite{G} for the proof.

Although Maruyama works in a arbitrary algebraically closed field,
we restrict ourselves to $\C$.

\begin{theorem}\cite[Theorem 2]{MARA}
Let $\pi\colon X \longrightarrow C$ be a ruled surface with invariant
$e$. Then:
\begin{enumerate}
\item If $e<0$, then $\aut_C(X) \simeq \Delta$.
\item If $e\geq 0$, $X$ is indecomposable and if $C_0$ is the unique minimal section, then
\[
\aut_C(X) \simeq \left\{ \left(\left(\begin{array}{cc}
1 & 0 \\
0 & 1
\end{array}\right),
\left(\begin{array}{cc}
1 & t_1 \\
0 & 1
\end{array}\right), \dots,
\left(\begin{array}{cc}
1 & t_r \\
0 & 1
\end{array}\right)\right) \,
\left|\begin{array}{c} t_i \in \C
\end{array} \right.
\right\}
\]
where $r = h^0(C,\OO_C(-\pi(C_0^2)) ) = h^0(C, \det(E)^{-1}\otimes
L_{C_0}^{\otimes 2})$, $X\simeq \PP(E)$.
\item If $X$ is decomposable and if $X$ does not carry two minimal sections, $C_0$ and $C_1$ such that $\pi(C_0^2) = \pi(C_1^2)$ (as divisor classes), then
\[
\aut_C(X) \simeq \left\{ \left(\left(\begin{array}{cc}
\alpha & 0 \\
0 & 1
\end{array}\right),
\left(\begin{array}{cc}
\alpha & t_1 \\
0 & 1
\end{array}\right), \dots,
\left(\begin{array}{cc}
\alpha & t_r \\
0 & 1
\end{array}\right)\right) \,
\left|\begin{array}{c}
t_i \in \C \\
 \alpha \in \C^{\ast}
\end{array} \right.
\right\}
\]
where $r = h^0(C,\OO_C(-\pi(C_0^2)) )$.
\item If $X$ is decomposable, $X\not\simeq \PP^1 \times C$ and if $X$ has two distinct minimal sections, $C_0$ and $C_1$ such that $\pi(C_0^2) = \pi(C_1^2)$ (accordingly $e=0$), then
\[
\aut_C(X) \simeq \left\{ \left(\begin{array}{cc}
\alpha & 0 \\
0 & 1
\end{array}\right)
\, \left|\begin{array}{c}
 \alpha \in \C^{\ast}
\end{array} \right.
\right\} \bigcup \left\{ \left(\begin{array}{cc}
0 & \beta \\
1 & 0
\end{array}\right)
\, \left|\begin{array}{c}
 \beta \in \C^{\ast}
\end{array} \right.
\right\}
\]
\item If $X\simeq \PP^1 \times C$, then
\[
\aut_C(X) \simeq \pgl(2,\C)
\]
\end{enumerate}
\end{theorem}

From this classification we can easily extract the finite subgroups
in each case.

\begin{corollary}\label{finp1}
Let $\pi\colon X \longrightarrow C$ be a ruled surface with invariant $e$
and let $G< \aut_C(X)$ be a finite subgroup.
\begin{enumerate}
\item If $e<0$, then
\[
G \simeq \left(\Z/2\Z\right)^r
\]
for some $r\geq 0$.
\item If $e\geq 0$, $X$ is indecomposable, then $G$ is trivial.
\item If $X$ is decomposable and if $X$ does not carry two minimal sections, $C_0$ and $C_1$ such that $\pi(C_0^2) = \pi(C_1^2)$ (as divisor classes), then $G$ is cyclic with generator
\[
\left(\begin{array}{cc}
\alpha & \Gamma \\
0 & 1
\end{array}\right),
\]
where $\alpha$ is a root of the unity and $\Gamma \in \HH^0(C,\OO_C(-\pi(C_0^2)) )$. 
\item If $X$ is decomposable, $X\not\simeq \PP^1 \times C$ and if $X$ has two distinct minimal sections, $C_0$ and $C_1$ such that $\pi(C_0^2) = \pi(C_1^2)$ (accordingly $e=0$), then the elements of $G$ can have the following forms:
\[
 \left(\begin{array}{cc}
\alpha & 0 \\
0 & 1
\end{array}\right),
\left(\begin{array}{cc}
0 & \beta \\
1 & 0
\end{array}\right),
\]
where $\alpha$ and $\beta$ are roots of the unity.
\item If $X\simeq \PP^1 \times C$, then $G$ is a finite subgroup of $\pgl(2,\C)$, namely:
\begin{itemize}
\item cyclic groups;
\item dihedral groups;
\item the tetrahedral group isomorphic to the alternating group $\mathcal{A}_4$;
\item the octahedral group isomorphic to the symmetric group $\mathcal{S}_4$;
\item the icosahedral group isomorphic to the alternating group $\mathcal{A}_5$.
\end{itemize}
\end{enumerate}
\end{corollary}

\begin{proof}
The first case comes from the fact that $\Delta$ is a subgroup of
the $2$-torsion part of ${\rm Pic}(C)$: if $E \simeq E\otimes N$,
$\det(E) \simeq \det(E) \otimes N^{\otimes 2} $, hence $N^{\otimes
2}$ is trivial. 

For the second case, observe that any element of the
form
\[
\left(\begin{array}{cc}
1 & t \\
0 & 1
\end{array}\right)
\]
has infinite order, unless $t=0$. 

For the third case, let $\Gamma, \Sigma \in \HH^0(C,\OO_C(-\pi(C_0^2)))$. Fix a point $x\in C$ and homogeneous coordinates $(z:w)$ on the fiber over $x$. Define
\begin{align*}
T_x (z,w) & = (\alpha z + \Gamma(x)w: w),\\
S_x (z,w) & = (\beta z + \Sigma(x)w: w)
\end{align*}
where $\alpha, \beta \in \C^{\ast}\backslash {1}$. Then, 
\[
T_x^{-1} \circ S_x^{-1} \circ T_x \circ S_x (z,w) = \left(z+ \frac{\alpha-1}{\alpha\beta}\Sigma(x) - \frac{\beta-1}{\alpha\beta}\Gamma(x) :w\right).
\]
If we suppose that $T$ and $S$ belong to a finite group $G$, then 
\[
\frac{\alpha-1}{\alpha\beta}\Sigma(x) - \frac{\beta-1}{\alpha\beta}\Gamma(x) = 0, \, \forall x\in C,
\]
which implies that $T$ and $S$ commute. Hence, $G$ is abelian. Since any finite abelian subgroup of $\pgl(2,\C)$ is cyclic, so is $G$. Observe that for $\alpha \neq 1$ and $\Gamma \in \HH^0(C,\OO_C(-\pi(C_0^2)))$,
\[
T^n_x(z,w) = \left(\alpha^n z + \left(\sum_{i=0}^{n-1}\alpha^i\right)\Gamma(x) w : w\right) = \left(\alpha^n z + \left(\frac{\alpha^n-1}{\alpha -1}\right)\Gamma(x) w : w\right).
\]
Hence, $\alpha^n=1$ if and only if $T^n_x=id$, for all $x\in C$.

The fourth case is straightforward, we have to impose only that $\alpha$ and $\beta$ are roots of the unity. The fifth case is the well-known classification of finite subgroups of $\pgl(2,\C)$, see \cite{BLI} for example.
\end{proof}

\subsection{Bounds}
The classification of finite subgroups of automorphisms and the properties of foliations on ruled surfaces presented above are our main tools to establish the bounds. We will use Theorem \ref{p1tg} in order to split in two cases: whether the foliation has all singularities lying invariant fibers or not. Let us begin by the second one.

\begin{theorem}
Let $\pi: X\longrightarrow C$ be a $\PP^1$-bundle over a smooth curve
whose genus is $g\geq1$ and let $e$ denote the invariant of $X$. Let
$\F$ be a foliation on $X$ of bidegree $(a,b)$ such that:
\begin{enumerate}
\item $\aut(\F)$ is finite;
\item $\F$ has a singularity at $p\in X$ which lies on a fiber $F$ not $\F$--invariant;
\item $\mu(\F,p) = {\rm Z}(\F,F,p)$ if $p$ lies in a $\F$--invariant fiber $F$.
\end{enumerate}
Then:
\begin{itemize}
\item $|\aut(\F)| \leq (8g + 4)[(a+1)(2b-ae + 2 -2g) + 2-2g]$ if $e<0$;
\item $|\aut(\F)| \leq (a+1)(4g + 2)[(a+1)(2b-ae + 2 -2g) + 2-2g]$ if $e\geq 0$.
\end{itemize}
\end{theorem}

\begin{proof}
Take $G$ the stabilizer of $p$ in $\aut(\F)$. Then $G$ has index at
most
\[
c_2(TX\otimes K_{\F}) = (a+1)(2b-ae + 2 -2g) + 2-2g,
\]
the number of singularities. The projection $\pi$ induces a group homomorphism, hence an exact sequence:
\[
1 \longrightarrow K \longrightarrow G \longrightarrow H
\longrightarrow 1, 
\]
where $K = G \cap \aut_C(X)$ and $H<\aut(C)$. Since $G$ fixes $p$, $K$ also
fixes $p$ and $H$ fixes $\pi(p)$. In particular, $H$ is cyclic. We
have that $K$ is also cyclic. Indeed, in a small neighborhood of
$p$ we can choose coordinates $(x,y)$ with $p=(0,0)$ and apply Corollary \ref{fibloc} to the ruling.

Suppose that $e<0$. Then, by Corollary \ref{finp1}, $K < \Z/2\Z$. By
Wiman's bound (\ref{wim}), $|H| \leq 4g+2$. Therefore,
\[
|G| = |H||K| \leq 8g + 4.
\]

If $e\geq0$, we also have that $|H| \leq 4g+2$. The fiber $F$ that
contains $p$ is not $\F$--invariant, then, by Proposition
\ref{tgvec},
\[
|K| \leq \tg(\F,F,p) + 1 \leq a+1.
\]
Therefore,
\[
|G| = |H||K| \leq (a+1)(4g + 2).
\]
\end{proof}

When all singularities lie on invariant fibers, Theorem \ref{p1tg}
imposes restrictions on which foliations can occur. We give them the following bound:

\begin{theorem}
Let $X\longrightarrow C$ be a $\PP^1$-bundle over a smooth curve
whose genus is $g\geq 1$ and let $e$ denote the invariant of $X$. Let
$\F$ be a foliation on $X$ of bidegree $(a,b)$, $ab\neq 0$, such
that
\begin{enumerate}
\item $\aut(\F)$ is finite;
\item All singularities lie on $\F$--invariant fibers and satisfy $\mu(\F,q) = {\rm Z}(\F,F,q)$;
\end{enumerate}
Then
\begin{itemize}
\item  $|\aut(\F)|\leq 84(g-1)(a+1)$ if $g\geq 2$;
\item  $|\aut(\F)|\leq 6(a+1)b$  if $g=1$.
\end{itemize}
\end{theorem}

\begin{proof}
By Theorem \ref{p1tg}, we have that $e>0$, $b=(a+1)e$ since $ab\neq0$. The
tangency divisor between $\F$ and the ruling is composed of $b+2-2g$
fibers and the negative section $aC_0$. Moreover, for any fiber $F$
not $\F$--invariant, $F$ and $\F$ have a single tangent point of
multiplicity $a$ that lies on the intersection with $C_0$. Since
$\tg(\F, C_0) = b -(a+1)e = 0$, $C_0$ does not support any
singularity of $\F$.

As we have mentioned in the previous theorem, $\aut(\F)$ fits into  an
exact sequence
\[
1 \longrightarrow K \longrightarrow \aut(\F) \longrightarrow H
\longrightarrow 1
\]
where $K < \aut_C(X)$ and $H<\aut(C)$. By Corollary \ref{finp1}, $K$
is cyclic and fixes $C_0$ pointwise. For the fiber $F$ through $p$,
$F\cap C_0$ is a regular point of $\F$ fixed by $K$. Then Proposition \ref{tgvec} implies that
\[
|K| \leq \tg(\F, F,F\cap C_0) + 1 = a+1 = \frac{b}{e}.
\] 
It remains to bound the order of $H$ which is a finite subgroup of
$\aut(C)$. If $g\geq 2$, then
\[
|H| \leq 84(g-1)
\]
by the Hurwitz's Theorem. If $g=1$, $H < \Z/n\Z \ltimes \Z/m\Z$ where
$n\leq 6$, $\Z/n\Z$ has a fixed point and $\Z/m\Z$ is a group
generated by translations. We have that $H$ permutes the images of
the $b+2-2g = b$ invariant fibers, then
\[
|H| \leq 6b = 6(a+1)e .
\]

Combining these bounds we have
\begin{enumerate}
\item $|\aut(\F)|\leq 84(g-1)(a+1)$ if $g\geq 2$ and
\item $|\aut(\F)|\leq 6(a+1)b$ if $g=1$.
\end{enumerate}

\end{proof}

\section{Bounds for Foliations on Non-ruled Surfaces}

Throughout this section we will analyze the automorphism groups of
foliations on surfaces that are not birationally ruled, that is,
they have nonnegative Kodaira dimension. In particular, such
surfaces do not support pencils of rational curves and this  will play an important role on our approach to bound the order of the groups. We begin with the simplest case: regular foliations. It turns out that they live naturally on general type surfaces.

\subsection{Regular Foliations of General Type}

Let $X$ be a smooth projective surface and $\F$ a foliation  of general type on $X$
 and suppose that the minimal model $(Y,\mathcal{G})$
is regular. In this situation, the Baum-Bott formulae express the
Chern numbers of $Y$ in terms of the foliation.

\begin{lemma}\label{reglem}
Let $\mathcal{G}$ be a regular foliation of general type on a smooth
compact surface $Y$, then
\[
c_1(Y)^2 > 2c_2(Y).
\]
\end{lemma}

\begin{proof}
Since $\mathcal{G}$ is regular it follows from  Baum-Bott formulas $(2.1)$ and $(2.2)$ that 
\begin{align*}
c_2(Y) &= K_X\cdot K_{\mathcal{G}} - K_{\mathcal{G}}^2,\\
c_1(Y)^2 &=2K_Y\cdot K_{\mathcal{G}} - K_{\mathcal{G}}^2.
\end{align*}
Therefore,
\[
c_1(Y)^2 -2c_2(Y) =  2K_Y\cdot K_{\mathcal{G}} - K_{\mathcal{G}}^2     -2K_X\cdot K_{\mathcal{G}} +2 K_{\mathcal{G}}^2= K_{\mathcal{G}}^2.
\] 
Since $\mathcal{G}$ is regular  and of general type, we have that
$K_{\mathcal{G}}$ is nef and  big. Hence
\[
0<K_{\mathcal{G}}^2 =c_1(Y)^2 -2c_2(Y).
\]
\end{proof}

\begin{theorem}
Let $\F$ be a foliation of general type on a smooth surface $X$. If
$\F$ is birationally regular, then
\[
|\bim(\F)|\leq (42K_{Y})^2,
\]
where $(Y,\G)$ is its (regular) minimal model.
\end{theorem}

\begin{proof}
Let $(Y,\G)$ be the minimal model of $(X,\F)$, which is regular by
hypothesis. It follows from Brunella's classification of regular foliations on
surfaces \cite[Th\'eor\`eme 2]{Bru1} that if $Y$ is not of general type,  then $c_1(Y)^2 = 2c_2(Y)$. Therefore, by lemma \ref{reglem} we conclude
that $Y$ is a surface of general type. Moreover, Brunella proves   in \cite[Corollaire 1]{Bru1}   
that  $Y$ must be minimal.

Since $\F$ is of general type it follows from \cite{Bru3}  that $\bim(\F) \simeq \aut(\G)$. Hence, Xiao's bound
(\ref{xiao}) implies that
\[
|\bim(\F)| = |\aut(\G)| \leq |\aut(Y)| \leq (42K_{Y})^2.
\]
\end{proof}

\subsection{Bounds for Singular Foliations}The strategy that we will follow to bound the order of the automorphisms groups in this context is to analyze the action on some very ample linear systems associated to the canonical bundle of the foliation. Then we will find some appropriate divisors to use the local theory. The following Lemma will help us to choose these linear systems.

\begin{lemma}\label{fuj}
Let $X$ be a complex  projective  surface and $L$ an ample divisor on
$X$. Then $|K_X + 4L|$ is very ample.
\end{lemma}

\begin{proof}
For any effective divisor $E$ in $X$,  we have $L\cdot E \geq 1$. In particular, we have that 
$nL\cdot E \geq n$ and $(nL)^2 \geq n^2$ and this implies that the
exceptional cases of  Reider's theorem \cite{Rei} do not occur
for $n\geq 4$.
\end{proof}

In \cite{CF}  the authors proved  that, in general,
foliations with ample canonical bundle have finite automorphism
group. There is only one family of exceptions.  They summarize this in
the following proposition:

\begin{proposition}\cite{CF}
 If $\F$ is a foliation on a smooth projective surface $X$ with $K_{\F}$ ample and ${\rm Aut}(\F)$ infinite,
 then up to a birational map, $\F$ is preserved by the flow of a vector field $v=v_{1}\oplus v_{2}$ on $\PP^{1}\times \PP^{1}$. Moreover, if $v$ is not tangent to a foliation by rational curves then $\F$ is given by a global vector field on $\PP^{1}\times \PP^{1}$.
\end{proposition}

Consequently, non-rational surfaces only support foliations with
ample canonical bundles whose automorphism groups are finite.

\begin{corollary}\label{nrat}
If $(X,\F)$ is a foliated surface such that $X$ is not rational and
$K_{\F}$ is ample, then  ${\rm Aut}(\F)$ is finite.
\end{corollary}

Moreover, if we suppose that $\F$ is reduced, then  $(X,\F)$ is minimal.
 Indeed, if  $(X,\F)$ is not minimal  there exists an $\F$--exceptional
curve $E$. Then 
\[
2\geq {\rm Z}(\F,E) = K_{\F}\cdot E -K_X\cdot E -E^2 = K_{\F}\cdot E +2,
\]
which would imply that $K_{\F}$ is not ample. Therefore, for reduced foliations with ample canonical bundle on non-ruled surfaces, it follows that $\bim(\F) = \aut(\F)$ and we have the following bound:

\begin{theorem}
Let $(X,\F)$ be a reduced foliated surface with ample
canonical bundle $K_{\F}$ such that $X$ is not ruled. Suppose that $\F$
has a singularity at $p\in X$ with eigenvalue $\lambda \neq -1$ which does not lie on a $\F$--invariant algebraic curve, then
\[
|{\rm Aut}(\F)| \leq c_2(T_X\otimes K_{\F})[((K_X + 5K_\F)\cdot(K_X + 4K_\F))^2 -1].
\]
\end{theorem}

\begin{proof}
It follows from Corollary \ref{nrat} that $\aut(\F)$ is finite. Let $G$ be the stabilizer of $p$ in $\aut(\F)$ which has index bounded by the number of singularities:
\begin{equation}\label{cnr0}
(\aut(\F):G) \leq c_2(T_X\otimes K_{\F}).
\end{equation}
By Proposition \ref{stab} we have that  $G$ is abelian and there exist coordinates $(x,y)$ on a small neighborhood $U\ni p$ such that $p=(0,0)$, $G$ is diagonal and $\F$ is given by a vector field $v$ with diagonal linear part.

The line bundle $K_X \otimes K_{\F}^{\otimes 4}$ is $G$--invariant,
since $K_X$ and $K_{\F}$ are $G$--invariant. By Lemma
\ref{fuj}, we also have that  $K_X \otimes K_{\F}^{\otimes 4}$  is very ample, hence
\[
N+1:=\dim \HH^0(X,K_X \otimes
K_{\F}^{\otimes 4}) \geq 3.
\]
Since $G$ is abelian, the induced action on $\HH^0(X,K_X \otimes 
K_{\F}^{\otimes 4})$ is diagonalizable. There exist a basis of $\HH^0(X,K_X \otimes K_{\F}^{\otimes 4})$ composed by $G$--semi-invariant sections $s_0, \dots, s_N$. That is,
\begin{equation}\label{char}
s_i\circ \varphi =\chi_i (\varphi) s_i, \, i=0,\dots, N,
\end{equation}
for every $\varphi\in G$ and $\chi_i$ is the corresponding character. Consider the map 
\[
\Gamma\colon X \longrightarrow \PP^N
\]
defined by $\Gamma(x) = (s_0(x): \dots: s_N(x))$. Since $K_X \otimes K_{\F}^{\otimes 4}$ is very ample, $\Phi$ is an embedding. This implies, in particular, that $s_i(p)\neq 0$ for some $i$. Suppose, without loss of generality, that $s_0(p)=1$. We can assume that $U\subset \{s_0 \neq 0\}$ and it follows that the restriction of $\Gamma$ to $U$ is expressed by
\[
\left.\Gamma\right|_U(x,y) = (t_1(x,y), \dots, t_N(x,y)),
\]
where $s_i = t_i s_0$. Since $\Gamma$ is an embedding, then $D\Gamma_0$ has rank two. Thus,  we can assume that 
\[
A(x,y)\ceq\begin{pmatrix}
\dfrac{\partial t_1}{\partial x} & \dfrac{\partial t_1}{\partial y} \\[6pt]
\dfrac{\partial t_2}{\partial x} & \dfrac{\partial t_2}{\partial y}
\end{pmatrix}
\]
is invertible at the origin. In particular, if $\Gamma(p)=(1:a_1: \dots: a_N)$, then the curves 
\[
C_1 =\{s_1 -a_1s_0 = 0 \}\  \, {\rm and  } \, \ C_2=\{s_2 -a_2s_0 = 0\}
\]
are smooth and transverse at $p$. We have that the curves $C_1$ and $C_2$ are $G$--invariant. Indeed, let $\varphi$ be an element of $G$. The equation (\ref{char}) implies that
\[
0 = s_i(p) -a_is_0(p) = s_i(\varphi(p)) -a_is_0(\varphi(p)) = \chi_i(\varphi)s_i(p) -a_i\chi_0(\varphi)s_0(p),
\]
for $i=1,\dots, N$. Since $s_0(p)=1$, it follows that either $a_i=0$ or $\chi_i(\varphi)=\chi_0(\varphi)$. In both cases, 
\[
(s_i -a_is_0)\circ \varphi = \chi_i(\varphi)(s_i -a_is_0).
\]

Suppose that $\varphi$ is given by
\[
\varphi(x,y) = (l^ax,l^by), 
\]
where $l^n =1$ primitive and $\gcd(a,b,n)=1$. Then the equation (\ref{char}) also implies that
\begin{equation}\label{char2}
\chi_0(\varphi) A(0) \begin{pmatrix}
l^a & 0 \\
0 & l^b
\end{pmatrix} = 
\begin{pmatrix}
\chi_1(\varphi) & 0 \\
0 & \chi_2(\varphi)
\end{pmatrix}
A(0).
\end{equation}
We have two distinct cases:
\begin{enumerate}
\item The curves $C_1$ and $C_2$ are tangent to the axes at the origin; 
\item At least one of the curves $C_i$ is transverse to the axes at the origin.
\end{enumerate}

First case: \\
Up to exchanging $x$ with $y$, we may suppose that $C_1$ is tangent to $\{y=0\}$ and $C_2$ is tangent to $\{x=0\}$. Then
\[
\dfrac{\partial t_1}{\partial y}(0) = \dfrac{\partial t_2}{\partial x}(0)=0.
\]
Since by hypothesis  the singularity at $p\in Sing(\F)$  has eigenvalue $\lambda \neq -1$ which does not lie on a $\F$--invariant algebraic curve, then   $C_1$ and $C_2$ are not $\F$--invariant.

The Proposition \ref{fibtg}, applied to each curve, implies that $G$ is cyclic and has generators
$$
\psi_1(x,y)   = (\zeta^{k_1}x,\zeta y) 
$$
and 
$$
\psi_2(x,y)   = (\zeta x,\zeta^{k_2} y), 
$$
where $\zeta^{|G|}=1$, with $\zeta$ primitive and $k_i=\tg(\F,C_i,p)$, $i=1,2$ (since $\lambda \neq -1$). 

Since
 $\psi_1$ and  $\psi_2$ are generators of the cyclic group $G$,  we have that  $\psi_1^{k_2} = \psi_2$.  Hence 
\[
k_1 k_2 \equiv 1 \pmod{|G|}.
\]
Since these curves are tangent to the axes and the linear part of $v$ is diagonal, we have $k_1,k_2 \geq 2$. Hence
\begin{align}\label{cnr1}
|G| \leq k_1 k_2 -1 & \leq [(K_\F +C_1)\cdot C_1 ][(K_\F +C_2)\cdot C_2] -1 \nonumber \\
& = [(K_X + 5K_\F)\cdot(K_X + 4K_\F)]^2 -1,
\end{align}
since by construction  $C_i= K_X + 4K_\F$, for $i=1,2$.

Second case:
If either $C_1$ or $C_2$ is transverse to the axes, then the equation (\ref{char2}) implies that 
\[
\chi_0(\varphi)l^a=\chi_0(\varphi)l^b=\chi_1(\varphi) =\chi_2(\varphi).
\]
Therefore, $G$ is cyclic and is generated by
\[
\psi(x,y)=(\zeta x, \zeta y),
\]
where $\zeta^{|G|}=1$ with $\zeta$ primitive, and there exist $c_1,c_2 \in \C$ such that
\[
C' =   \{  c_1 (s_1 -a_1s_0)+  c_2(s_2 -a_2s_0 )= 0\}    
\]
is $G$--invariant, tangent to $\{x=0\}$ and smooth at $p$.  Hence, by the Proposition \ref{fibtg}, we have 
\[
1 \neq \tg(\F, C',p)\equiv 1 \pmod{|G|},
\]
which implies that
\begin{equation}\label{cnr2}
|G| \leq \tg(\F, C',p) -1 \leq \tg(\F, C') -1 = (K_X + 5K_\F)\cdot(K_X + 4K_\F)-1.
\end{equation}

To conclude, observe that $\kod(X)\geq 0$ hence 
\[
(K_X + 5K_\F)\cdot(K_X + 4K_\F) = g-1 + 3K_\F\cdot (K_X + 4K_\F) \geq 12K_\F^2 \geq 12,
\]
where $g$ is the arithmetic genus of a curve in the linear system $|K_X+4K_\F|$. Then these bounds are not trivial. 
Comparing the inequalities (\ref{cnr1}) and (\ref{cnr2})   with (\ref{cnr0}), we have 
\[
|\aut(\F)| \leq c_2(T_X\otimes K_{\F})[((K_X + 5K_\F)\cdot(K_X + 4K_\F))^2 -1].
\]
\end{proof}

\subsection*{Acknowlegments}We are grateful to Omegar Calvo Andrade, Alex Massarenti and Bruno Sc\'ardua for interesting conversations.
The first named author   was partially supported by CAPES, CNPq and Fapesp-2015/20841-5 Research  Fellowships. 
The  second  named  author was partially supported by CAPES and   is grateful  to  Instituto de Matemática e Estatística--Universidade Federal Fluminense for hospitality.
Finally, we would like to thank the referee by the suggestions, comments and improvements
to the exposition.

\end{document}